\documentclass[12pt]{article}
\usepackage{graphicx}
\usepackage{amsmath, amssymb, amsthm}
\usepackage{bbm}
\usepackage{enumerate}
\usepackage{tikz-cd}
\usepackage{cancel}
\usepackage{bm}
\usepackage{mathtools}
\usepackage[left=25mm, right=25mm]{geometry}

\newtheorem{theorem}{Theorem}[section]
\newtheorem{proposition}[theorem]{Proposition}
\newtheorem{lemma}[theorem]{Lemma}
\newtheorem{corollary}[theorem]{Corollary}

\theoremstyle{definition}
\newtheorem{definition}[theorem]{Definition}
\newtheorem{example}[theorem]{Example}
\theoremstyle{remark}
\newtheorem{remark}[theorem]{Remark}

\numberwithin{equation}{section}

\newcommand{\R}{\mathbb{R}}
\newcommand{\Z}{\mathbb{Z}}
\newcommand{\Q}{\mathbb{Q}}
\newcommand{\C}{\mathbb{C}}
\newcommand{\PP}{\mathbb{P}}

\newcommand{\D}{\mathcal{D}}
\newcommand{\DD}{\mathbb{D}}
\newcommand{\spec}{\operatorname{Spec}}

\newcommand{\dR}{\mathrm{dR}}

\newcommand{\FT}{\operatorname{FT}}
\newcommand{\Cancel}{\operatorname{\mathbf{Cancel}}}
\newcommand{\rank}{\operatorname{rank}}
\newcommand{\fin}{\widetilde{\mathrm{fin}}}

\newcommand{\1}{\mathbbm{1}}
\newcommand{\hyp}{\mathrm{Hyp}}

\newcommand{\Ker}{\operatorname{Ker}}

\title{GKZ Systems and Gauss--Manin Connections of Families of Projective Hypersurfaces}
\author{Ryo Negishi}
\date{}

\begin{document}
\maketitle
\begin{abstract}
Periods of projective hypersurfaces satisfy GKZ hypergeometric systems, which may have solutions that do not arise as periods. For deformations of Fermat hypersurfaces, we describe the Gauss--Manin connection in terms of images of contiguity operators between GKZ systems. These images are also realized as Fourier transforms of minimal extensions of Kummer modules.
\end{abstract}
\section{Introduction}

For a smooth projective morphism \(p:Y \rightarrow S\) of smooth complex algebraic varieties, the systems of differential equations satisfied by the periods of its fibers are called Picard--Fuchs systems and are a classical subject of study. Picard--Fuchs systems have an intrinsic algebraic description in terms of the Gauss--Manin connection. Katz and Oda \cite[Theorem~1]{KO} constructed a canonical integrable connection, called the Gauss--Manin connection, on the relative de Rham cohomology \(H^q_{\dR}(Y/S)=R^qp_*\Omega^\bullet_{Y/S}\):
\[
  \nabla^{\mathrm{GM}}:H^q_{\dR}(Y/S) \rightarrow H^q_{\dR}(Y/S)\otimes_{\mathcal O_S}\Omega^1_S.
\]
Integration over locally constant families of homology classes intertwines \(\nabla^{\mathrm{GM}}\) with differentiation of periods. Thus differential relations among cohomology classes yield Picard--Fuchs systems for their periods. In algebraic geometry, Gauss--Manin connections are used to study variations of Hodge structure and monodromy. The variation of Hodge structure associated with a family of Calabi--Yau varieties is closely related to mirror symmetry. In particular, solutions of Picard--Fuchs equations are used to construct mirror maps and compute Yukawa couplings \cite{CK}. Moreover, the Gauss--Manin connection also plays an important role in arithmetic geometry by providing information about the Frobenius action on \(p\)-adic cohomology.

The Gauss--Manin connections of certain one-parameter families of projective hypersurfaces are known to be described by generalized hypergeometric differential equations, e.g. \cite{K09,N}. More generally, periods of a family of projective hypersurfaces are known to satisfy systems of linear partial differential equations called GKZ systems \cite{GKZ90}, introduced by Gel'fand, Kapranov, and Zelevinsky. GKZ systems satisfied by periods, however, may have solutions that do not arise from the periods. For example, consider the four-parameter family defined by
\[
  \lambda_1X_1^3+\lambda_2X_2^3+\lambda_3X_3^3 + \lambda_4X_1X_2X_3=0.
\]
GKZ systems are determined by an integer matrix \(A \in M_{n \times m}(\Z)\) and a parameter vector \(\bm{\beta} \in \C^n\) (Definition \ref{def:GKZ}). Then the corresponding GKZ system is given by
\[
  A=\begin{pmatrix}
    3&0&0&1\\
    0&3&0&1\\
    0&0&3&1
  \end{pmatrix},
  \quad \bm{\beta}=(-1,-1,-1)^T.
\]
With \(\partial_i=\partial/\partial\lambda_i\), the GKZ system associated with \(A\) and \(\bm{\beta}\) is given explicitly by
\[
  \begin{cases}
    (\partial_1\partial_2\partial_3-\partial_4^3)f=0,\\
    (3\lambda_i\partial_i+\lambda_4\partial_4+1)f=0 \quad (i=1,2,3).
  \end{cases}
\]
Restricting to \((\lambda_1,\lambda_2,\lambda_3,\lambda_4)=(1,1,1,\lambda)\) gives the third-order equation
\[
  27\partial_\lambda^3+(\theta_\lambda+1)^3)f=0, \quad \partial_\lambda=\frac{d}{d\lambda},\quad \theta_\lambda = \lambda\partial_\lambda.
\]
Hence the periods of the Hesse cubic curves
\[
  X_1^3 + X_2^3 + X_3^3 + \lambda X_1X_2X_3=0 \quad (\lambda^3 \neq -27)
\]
satisfy the preceding differential equation. On the other hand, the Picard--Fuchs equation should be second-order since this is a non-isotrivial family of elliptic curves. Thus, it is natural to ask how to determiine the Gauss--Manin connection from the corresponding GKZ system. This question has been studied, for example, in \cite{semi-periods,HLY,Nagano,Zhou}.

We consider the family \(Y\subset\mathbb P^{n-1}\times S\), with \(S=\mathbb A^m_\C = \spec\C[\lambda_1, \ldots, \lambda_m]\), defined by
\begin{equation}\label{eq:intro-family}
  F_A=\sum_{i=1}^n\lambda_iX_i^d + \sum_{j=n+1}^m\lambda_jX^{\bm{a}_j} = 0, \quad A=\bigl[dI_n\mid\bm{a}_{n+1}\ \cdots\ \bm{a}_m\bigr] \in M_{n \times m}(\Z),
\end{equation}
where \(d,m\geq n\geq2\), each \(\bm{a}_j\) has nonnegative integer entries, and the sum of its entries is \(d\). Let \(U\subset S\) be the largest open subset over which \(Y\to S\) is smooth, and let \(Y_U = Y \times_S U\). The finite abelian group \(\Gamma_A/\Delta\), defined in Section~\ref{sec:fermat-de-rham}, acts on the family. This action gives the eigendecomposition of the primitive part of de Rham cohomology \(H^{n-2}_{\dR}(Y_U/U)^{\mathrm{prim}}\). Our main result identifies every nonzero eigenspace of \(H^{n-2}_{\dR}(Y_U/U)^{\mathrm{prim}}\) with the restriction to \(U\) of the image of contiguity operators between GKZ systems. Each such image is also realized as the Fourier transform of the minimal extension of a suitable Kummer module (Proposition \ref{prop:contiguity-image}, Theorem \ref{thm:gauss-manin-gkz}).

More precisely, let \(\widetilde{A}\) be a normalization of \(A\) so that \(\Z\widetilde{A}=\Z^n\) and consider the monomial map
\[
  \phi^{\widetilde A}:\mathbb{G}_m^n\longrightarrow\mathbb{A}^m_\C; t\longmapsto(t^{\widetilde{\bm{a}}_1},\ldots,t^{\widetilde{\bm{a}}_m}),
\]
where \(\widetilde{\bm{a}}_j\) is the \(j\)-th column of \(\widetilde{A}\). Results of Schulze--Walther and Reichelt \cite{SW09,reichelt} describe a GKZ system associated with \(A\) and \(\bm{\beta} = (\beta_1, \ldots, \beta_n) \in \C^n\), under suitable conditions on \(\bm{\beta}\), as Fourier transforms of direct images \(\FT(\phi^{\widetilde{A}}_+\mathcal{K}_\beta), \FT(\phi^{\widetilde{A}}_!\mathcal{K}_\beta)\) of the Kummer module
\[
  \mathcal{K}_{\bm{\beta}} := \D_{\mathbb{G}_m^n}/\sum_{i=1}^n\D_{\mathbb{G}_m^n}(t_i\partial_{t_i}-\beta_i).
\]
Using the theory of algebraic Dwork cohomology due to Dimca, Maaref, Sabbah, and Saito \cite{DMSS}, we can construct a surjection from each nonzero eigenspace of \(H^{n-2}_{\dR}(Y_U/U)^{\mathrm{prim}}\) onto the restriction to \(U\) of the Fourier transform of the minimal extension \(\FT(\phi^{\widetilde{A}}_{!+}\mathcal{K}_{\bm{\beta}})\), for a suitable parameter \(\bm{\beta}\). By Proposition~\ref{prop:contiguity-image}, this Fourier transform is isomorphic to the image of a contiguity operator between GKZ systems. We then show that this surjection is an isomorphism using a holonomic-rank estimate based on the Gr\"obner deformation techniques developed in \cite{SST}.

We obtain a more explicit description of the Gauss--Manin connection for the one-parameter family
\begin{equation}\label{eq:one-parameter_family}
  X_1^d+\cdots+X_n^d+\lambda X_1^{a_1}\cdots X_n^{a_n}=0, \quad a_i\geq0,\quad \sum_{i=1}^n a_i=d.
\end{equation}
For this family, the restrictions of the GKZ systems to one variable become, after a change of variable, generalized hypergeometric differential equations of the form
\[
 \left(\prod_{j=1}^{N}(\theta_t+\nu_j-1) - t\prod_{i=1}^{N}(\theta_t+\mu_i)\right)f(t)=0, \quad \theta_t=t\frac{d}{dt}.
\]
Under this identification, the contiguity maps appearing in our main theorem correspond to contiguity maps for these hypergeometric equations. Then the resulting image of the contiguity operators is the hypergeometric differential equations obtained by removing all pairs of parameters satisfying \(\mu_i\equiv\nu_j\bmod \Z\).

Our work is closely related to \cite{AM}. Asakura and Matsubara-Heo describe the subspace of the solutions of the GKZ system spanned by periods of the family \eqref{eq:intro-family}. They use this description to study the periods of the limiting mixed Hodge structure.

The paper is organized as follows. Section~2 recalls the required \(\D\)-module operations and facts about GKZ systems. Section~3 constructs the maps from GKZ modules to Dwork cohomology and proves the main theorem. Section~4 establishes the holonomic-rank estimate used in that proof. Section~5 treats the case of the one-parameter family given by (\ref{eq:one-parameter_family}).

\section{Preliminaries}
\subsection{Operations on \(\D\)-modules}

Let \(S\) be a smooth algebraic variety over \(\C\). Throughout this paper, \(\D_S\)-module means \(\mathcal{O}_S\)-quasi-coherent left \(\D_S\)-module. We denote the bounded derived category of \(\D\)-modules (resp. holonomic \(\D\)-modules) by \(D^b(\D_S)\) (resp. \(D^b_h(\D_S)\)). When \(S\) is an affine variety, we identify \(\mathcal{D}_S\)-modules with their global sections since a smooth affine variety is D-affine \cite[Proposition 1.4.3]{HTT}.

Let \(f:X\to Y\) be a morphism of connected smooth complex algebraic varieties and let \(d_X=\dim X, d_Y=\dim Y\). We define the transfer bimodules by
\[
\mathcal D_{X\rightarrow Y} :=\mathcal{O}_X \otimes_{f^{-1}\mathcal{O}_Y} f^{-1}\D_Y, \quad \D_{Y\leftarrow X} := \omega_X \otimes_{\mathcal{O}_X} \mathcal D_{X\rightarrow Y} \otimes_{f^{-1}\mathcal{O}_Y}f^{-1}\omega_Y^{-1},
\]
where \(\omega_X = \Omega_X^{d_X}, \omega_Y = \Omega_Y^{d_Y}\).

For \(M \in D^b(\D_X)\), the direct image of \(M\) under \(f\) is defined by
\[
f_+M := Rf_*(\D_{Y\leftarrow X}\overset{L}{\otimes}_{\D_X}M).
\]
For \(N\in D^b(\D_Y)\), the inverse image of \(N\) under \(f\) is defined by
\[
f^+N := \D_{X\rightarrow Y}\overset{L}{\otimes}_{f^{-1}\D_Y}f^{-1}N.
\]

The holonomic duality functor is defined by
\[
\mathbb D_X(M) := R\mathcal{H}om_{\D_X}(M, \D_X)\otimes_{\mathcal{O}_X}\omega_X^{-1}[d_X].
\]
It induces an equivalence \(\DD_X : D_h^b(\mathcal D_X) \rightarrow D_h^b(\mathcal D_X)^{\mathrm{op}}\), and, for a holonomic \(\mathcal D_X\)-module \(M\), \(\DD_X(M)\) is again concentrated in degree zero.

We further define the exceptional direct and inverse image functors by
\[
f_! := \DD_Y\circ f_+\circ\DD_X \text{ and } f^! := \DD_X\circ f^+\circ\DD_Y.
\]

If \(f:X\hookrightarrow Y\) is an affine locally closed immersion and \(M\) is a holonomic \(\D_X\)-module, then \(f_!M\) and \(f_+M\) are holonomic \(\D_Y\)-modules concentrated in degree zero. The minimal extension \(f_{!+}M\) is the image of the canonical morphism \(f_!M\rightarrow f_+M\) \cite[Definition~3.4.1]{HTT}.

  Let \(S = \mathbb{A}^m_\C = \spec \C[\lambda_1, \ldots, \lambda_m]\) and let \(M\) be a \(\D_S\)-module. The \emph{Fourier transform} \(\FT(M)\) of \(M\) is a \(\D_S\)-module defined by
  \begin{align*}
    &\FT(M) = M \text{ as a }\C\text{-vector space},\\
    &\lambda_i\bullet x := -\partial_ix \quad \partial_i\bullet x:=  \lambda_ix \quad \text{for any } x\in M.
  \end{align*}
  For an object \(M^\bullet \in D^b(\D_S)\), the Fourier transform \(\FT(M^\bullet)\) of \(M^\bullet\) is defined termwise. Since \(\FT\circ \FT\) is the pullback by the involution \(\lambda \mapsto -\lambda\), Fourier transformation gives a categorical equivalence.

\subsection{Definition and properties of GKZ systems}\label{sec:definition_gkz}
Let \(n \leq m\) be positive integers and let \(A = \begin{pmatrix} \bm{a}_1 & \cdots & \bm{a}_m \end{pmatrix}\in \Z^{n \times m}\) be a full-rank integer matrix.
We write \(\1=(1,\ldots,1)\in\Z^n\).

\begin{definition}\label{def:GKZ}
Let
\[
  D_m:=\C\langle \lambda_1,\ldots,\lambda_m, \partial_1,\ldots,\partial_m\rangle = \D_{\mathbb{A}^m_\C}
\]
be the Weyl algebra. We use the multi-index notation \(\lambda^{\bm{u}}=\prod_{j=1}^m\lambda_j^{u_j}\) and \(\partial^{\bm{u}}=\prod_{j=1}^m\partial_j^{u_j}\) for \(\bm{u}=(u_1,\ldots,u_m)\in\Z_{\geq0}^m\).
The toric ideal \(I_A\) associated with \(A\) is the ideal of \(\C[\partial_1,\ldots,\partial_m]\) generated by
\[
  \partial^{\bm{u}} - \partial^{\bm{v}}, A\bm{u} = A\bm{v},\; \bm{u}, \bm{v} \in \Z_{\geq 0}^m.
\]
We put
\[
  E_i = \sum_{j=1}^{m} a_{ij}\lambda_j\partial_j \quad (i=1,\ldots,n).
\]
For \(\bm{\beta} = (\beta_1, \ldots, \beta_n) \in \C^n\), let \(H_A(\bm{\beta})\) be the left ideal of \(D_m\) generated by \(I_A\) and \(E_1 - \beta_1, \ldots, E_n - \beta_n\). Then the GKZ system with parameter \(\bm{\beta}\) is the \(D_m\)-module
\[
  M_A(\bm{\beta}) := D_m/H_A(\bm{\beta}) = D_m/\left(\sum_{i=1}^nD_m(E_i - \beta_i)+D_m I_A\right).
\]
\end{definition}

We write \(\Z A := \Z \bm{a}_1 + \cdots + \Z\bm{a}_m\).
In the remainder of this section, we assume that \(\Z A = \Z^n\).

\begin{definition}\label{def:positive_homogenious}
  We say that the semigroup \(\Z_{\geq 0} A\) is positive if \(\Z_{\geq 0} A \cap (-\Z_{\geq 0} A)=\{0\}\).

  We say that \(A\) is homogeneous if there exists a group homomorphism \(h:\Z^n \longrightarrow \Z\) such that \(h(\bm{a}_i)=1\) for all \(i=1,\ldots,m\).
\end{definition}
Let \(\phi^A:\mathbb{G}_{m}^n \rightarrow \mathbb{A}^m_\C\) be the morphism corresponding to the ring homomorphism
\[\C[\lambda_1,\ldots,\lambda_m] \rightarrow \C[t_1^{\pm 1},\ldots,t_n^{\pm 1}];\quad \lambda_i \mapsto t^{\boldsymbol{a}_i}.\]
For any \(\bm{\gamma} = (\gamma_1, \ldots, \gamma_n) \in \C^n\), we define the Kummer module as
\[
  \mathcal{K}_{\bm{\gamma}} := \D_{\mathbb{G}_m^n}/\D_{\mathbb{G}_m^n}(t_1\partial_1 - \gamma_1, \ldots, t_n\partial_n - \gamma_n).
\]

\begin{proposition}[{\cite[Corollary~3.7]{SW09}}]\label{thm:GKZ=FT+}
  Assume that \(\Z_{\geq 0} A\) is positive. Let \(\mathrm{sRes}(A)\subset \C^n\) be the set defined in \cite[Definition~3.5]{SW09}. If \(\bm{\beta} \notin \mathrm{sRes}(A)\),
  \[
    \FT(\phi^A_+\mathcal{K}_{-\bm{\beta}-\1}) \cong M_A(\bm{\beta}).
  \]
\end{proposition}

We omit the definition of \(\mathrm{sRes}(A)\) and only use the following property.
\begin{proposition}[{\cite[Lemma 1.10]{reichelt}}]\label{thm:sRes}
   There exists \(\bm{\delta}_A \in \Z_{\geq 0} A\) such that
  \[
    (\bm{\delta}_A + \R_{\geq 0} A) \cap \mathrm{sRes}(A) = \emptyset.
  \]
\end{proposition}

Similar results also hold for \(\phi^A_!\).

\begin{proposition}[{\cite[Proposition 1.14]{reichelt}}]\label{thm:GKZ=FT!}
  Assume that \(A\) is homogeneous. Let \(\DD\mathrm{sRes}(A)\subset \C^n\) be the set defined in \cite[p. 920]{reichelt}. Then
  \[
    \FT(\phi^A_!\mathcal{K}_{-\bm{\beta}-\1}) \cong M_A(\bm{\beta})
  \]
  for all \(\bm{\beta} \in \Q^n \setminus (\DD \mathrm{sRes}(A) \cap \Q^n)\).
\end{proposition}
\begin{proposition}[{\cite[p.920]{reichelt}}]\label{thm:DsRes}
Let \((\R_{\geq 0}A)^\circ\) be the interior of the positive cone \(\R_{\geq 0} A = \R_{\geq 0}\bm{a}_1 + \cdots + \R_{\geq 0}\bm{a}_m\). Then
      \[
    \DD \mathrm{sRes}(A) \cap -(\R_{\geq 0}A)^\circ = \emptyset.
  \]
\end{proposition}
\begin{remark}
The notation \(\mathcal{O}_{T'}y^\beta\) in \cite{reichelt}
corresponds to \(\mathcal{K}_{\bm{\beta}}\) in our notation.
Theorem~1.9 of loc.~cit. asserts, in our notation, that
\[
  M_A(\bm{\beta})
  \cong
  \FT(\phi^A_+\mathcal{K}_{\bm{\beta}}),
\]
however this is incorrect. Indeed, the Kummer module used in \cite[Corollary~3.7]{SW09} is \(\mathcal{K}_{-\bm{\beta}-\1}\).
The same error occurs in subsequent statements, in particular, \cite[Proposition~1.14]{reichelt}.
\end{remark}

\section{The Gauss--Manin Connection of a Family of Projective Hypersurfaces}
\subsection{Relative de Rham Cohomology of the Fermat deformation}\label{sec:fermat-de-rham}

Let \(d,m,n\) be integers satisfying \(d,m\geq n\geq 2\).
For each $j=n+1,\ldots,m$, let
\[
\bm{a}_j
=
\begin{pmatrix}
a_{1j} & \cdots & a_{nj}
\end{pmatrix}^{T}
\in \mathbb{Z}_{\geq 0}^{n}
\]
be a column vector satisfying \(\sum_{i=1}^{n}a_{ij}=d\). We also set
\[
\bm{a}_j=d\bm{e}_j,
\qquad
j=1,\ldots,n,
\]
where \(\bm{e}_1,\ldots,\bm{e}_n\) denote the standard basis vectors of \(\Z^{n}\).

Set \(S = \mathbb{A}^{m}_{\C} = \spec\C[\lambda_1, \ldots, \lambda_m], X = \mathbb{P}^{n-1}_S = \operatorname{Proj}_{S}\mathcal{O}_{S}[X_1,\ldots,X_n].\) Let \(Y\subset X\) be the family of projective hypersurfaces defined by the homogeneous polynomial
\begin{equation}
\label{eq:fermat-deformation}
F_A = \sum_{j=1}^{m}\lambda_jX^{\bm{a}_j} = \sum_{i=1}^{n}\lambda_iX_i^d + \sum_{j=n+1}^m \lambda_jX_1^{a_{1j}}\cdots X_n^{a_{nj}}
\end{equation}
of degree \(d\), that is, \(Y = \operatorname{Proj}\left(\mathcal{O}_S[X_1,\ldots,X_n]/(F_A)\right)\). We refer to \(Y/S\) as the Fermat deformation.

Let \(i:Y\hookrightarrow X\) and \(p:X\longrightarrow S\) denote the closed embedding and the projection, respectively. We define the primitive part of $H^i p_+R\Gamma_Y\mathcal{O}_X$ by
\[
(H^i p_+R\Gamma_Y\mathcal{O}_X)^{\mathrm{prim}} := \Ker(H^i p_+R\Gamma_Y\mathcal{O}_X\longrightarrow H^i p_+\mathcal{O}_X).
\]
The natural morphism on the right is surjective for
\(i\neq 1-n\) and is zero for \(i=1-n\).

We review the algebraic Dwork cohomology developed in \cite{DMSS}. 
For \(0\leq i\leq n\), define
\[
\Omega^i(\mathbb{A}^{n}_S/S)^{(d)} := \bigoplus_{\substack{I\subset\{1,\ldots,n\},\, |I|=i\\\bm{\nu}\in\mathbb{Z}_{\geq0}^{n},\,|\bm{\nu}|+|I|\equiv0\;(\mathrm{mod}\ d)
}}\mathcal{O}_{S}X^{\bm{\nu}} dX_I,
\]
where, for \(I=\{j_1<\cdots<j_i\}\), \(dX_I=dX_{j_1}\wedge\cdots\wedge dX_{j_i}\). We define the algebraic Dwork cohomology of \(Y/S\) by
\[
  H^i_{\mathrm{Dw}}(Y/S) := H^i\left(\Omega^\bullet\left(\mathbb{A}^{n}_S/S\right)^{(d)}, d+dF_A\wedge\right),
\]
and we denote the cohomology class of a cocycle \(\sum_{\nu}c_{\nu}X^{\bm{\nu}}dX_I\) by \(\left[\sum_{\nu}c_{\nu}X^{\bm{\nu}}dX_I \right]\). \(H^i_{\mathrm{Dw}}(Y/S)\) has a structure of a \(\D_S\)-module. In particular, the action of \(\partial_j\) on \(H^n_{\mathrm{Dw}}(Y/S)\) is given by \(\partial_j\bullet \left[ X^{\bm{\nu}}dX_I \right] = \left[ X^{\bm{\nu}+\bm{a}_j}dX_I \right]\).

\begin{theorem}
[{\cite[Thm.~0.4]{DMSS}}]
\label{thm:DMSS}
There is a canonical isomorphism of \(\D_S\)-modules
\[
\left(H^i p_+R\Gamma_Y\mathcal{O}_X\right)^{\mathrm{prim}} \cong H^{i+n-1}_{\mathrm{Dw}}(Y/S).
\]
\end{theorem}

Let \(\Delta_A\in\mathbb{C}[\lambda_1,\ldots,\lambda_m]\) denote the discriminant of \(F_A\), and let \(U = \spec \C[\lambda_1,\ldots,\lambda_m,\Delta_A^{-1}]\). Then the projection \(p_U\circ i_U:Y_U = Y \times_S U \longrightarrow U\) is smooth. The algebraic de Rham cohomology
\[
  H_{\dR}^{n-2}(Y_U/U) := R^{n-2}(p_U \circ i_U)_*\Omega^\bullet_{Y_U/U}
\]
has a \(\D_U\)-module structure given by the Gauss--Manin connection \cite{KO}, and
\[
H^1\left((p_U)_+R\Gamma_{Y_U}\mathcal{O}_{X_U}\right) \cong H_{\mathrm{dR}}^{n-2}(Y_U/U)
\]
as \(\mathcal{D}_U\)-modules \cite[Proposition~1.4]{DMSS}. We define the primitive part of the de Rham cohomology as the kernel of the Gysin homomorphism
\[
  H_{\dR}^{n-2}(Y_U/U)^{\mathrm{prim}} := \Ker\!\left(H_{\dR}^{n-2}(Y_U/U) \rightarrow H_{\dR}^{n}(X_U/U)\right).
\]
Since \(Y_U\) is a smooth divisor in \(X_U\), we have \(R\Gamma_{Y_U}\mathcal{O}_{X_U}[1] \cong (i_U)_+\mathcal{O}_{Y_U}\). Consequently, the preceding comparison induces an isomorphism
\begin{equation}\label{eq:comparison}
  H_{\dR}^{n-2}(Y_U/U)^{\mathrm{prim}} \cong H^{n}_{\mathrm{Dw}}(Y/S)|_U.
\end{equation}
\medskip

We next define an action of a finite abelian group on the cohomologies. Let
\[
  \Gamma_A = \left\{(\zeta_1,\ldots, \zeta_n) \in (\C^\times)^n \mid \prod_{i=1}^n\zeta_i^{a_{ij}} = 1 \text{ for all } j = 1, \ldots, m\right\}
\]
and let \(\Delta\) be the diagonal subgroup of \(\Gamma_A\). We set
\[
  H_{n,d} = \{\bm{\gamma}=(\gamma_1,\ldots,\gamma_n)\in\Z^n \mid \sum_{i=1}^n\gamma_i\equiv 0\bmod d\}.
\]
We identify the character group
\(\widehat{\Gamma_A/\Delta}:=\mathrm{Hom}_{\mathrm{Grp}}(\Gamma_A/\Delta,\C^\times)\)
with \(H_{n,d}/\Z A\) via the isomorphism
\begin{align*}
  H_{n,d}/\Z A &\longrightarrow \widehat{\Gamma_A/\Delta},\\
  \overline{\bm{\gamma}} &\longmapsto ((\zeta_1, \ldots, \zeta_n) \longmapsto\prod_{i=1}^n\zeta_i^{\gamma_i}).
\end{align*}

\(\Gamma_A\) acts on \(\mathbb{A}^n_S\) by
\begin{equation}\label{eq:action}
  (\zeta_1, \ldots, \zeta_n) \bullet (X_1, \ldots, X_n) := (\zeta_1X_1, \ldots, \zeta_nX_n).
\end{equation}
This action induces an action of the finite abelian group \(\Gamma_A/\Delta\) on the complex\\
\((\Omega^{\bullet}\left(\mathbb{A}^{n}_S/S\right)^{(d)}, d+dF_A\wedge)\), and hence on \(H^i_{\mathrm{Dw}}(Y/S)\). In particular, \(\Gamma_A/\Delta\) acts on \(H^n_{\mathrm{Dw}}(Y/S)\) by
\[
  (\zeta_1, \ldots, \zeta_n) \bullet  [X_1^{\nu_1}\cdots X_n^{\nu_n}dX_1\wedge \cdots \wedge dX_n] = \zeta_1^{\nu_1+1}\cdots \zeta_n^{\nu_n+1}[X_1^{\nu_1}\cdots X_n^{\nu_n}dX_1\wedge \cdots \wedge dX_n].
\]
Clearly, this action is compatible with the \(\D_S\)-module structure, it gives a decomposition
\[
H^n_{\mathrm{Dw}}(Y/S) = \bigoplus_{\overline{\bm{\gamma}} \in \widehat{\Gamma_A/\Delta}} H^n_{\mathrm{Dw}}(Y/S)_{\overline{\bm{\gamma}}}
\]
as a \(\D_S\)-module, where \(H^n_{\mathrm{Dw}}(Y/S)_{\overline{\bm{\gamma}}}\) denotes the component corresponding to \(\overline{\bm{\gamma}}\in H_{n,d}/\Z A\cong\widehat{\Gamma_A/\Delta}\).

The action in \eqref{eq:action} also induces an action of \(\Gamma_A/\Delta\) on \(Y_U\) over \(U\), and hence \(\Gamma_A/\Delta\) acts on \(H^{n-2}_{\dR}(Y_U/U)^{\mathrm{prim}}\). These actions are compatible with the isomorphism \eqref{eq:comparison}, therefore,
\[
  H_{\dR}^{n-2}(Y_U/U)^{\mathrm{prim}}_{\overline{\bm{\gamma}}} \cong H^{n}_{\mathrm{Dw}}(Y/S)_{\overline{\bm{\gamma}}}|_U.
\]

\begin{lemma}
   The holonomic rank of the component \(H^{n}_{\mathrm{Dw}}(Y/S)_{\overline{\bm{\gamma}}}\) is the number of elements of the set
\begin{equation}\label{eq:totally_nonzero_set}
    \mathfrak{I}_{\bm{\gamma}} = \{\overline{\bm{\beta}} = (\overline{\beta}_1, \ldots, \overline{\beta}_n) \in \Z A/d\Z^n \mid \beta_i + \gamma_i \not\equiv 0 \bmod d \text{ for all }i=1, \ldots, n\}.
\end{equation}
\end{lemma}

\begin{proof}
  The holonomic rank of \(H^{n}_{\mathrm{Dw}}(Y/S)_{\overline{\bm{\gamma}}}\) is the rank of \(H^{n-2}_{\mathrm{dR}}(Y_U/U)^{\mathrm{prim}}_{\overline{\bm{\gamma}}}\) as a the locally free \(\mathcal{O}_U\)-module. Then the claim can be proved in the same way as \cite[Lemma~3.1]{K09}.
\end{proof}

\subsection{GKZ systems associated with the Fermat deformation}
Let \(n, m, d\) and \(\bm{a}_1, \ldots, \bm{a}_m\) be defined as in the preceding section, and set
\begin{equation}
  \label{eq:matrix}
  A =
  \begin{pmatrix}
    \bm{a}_1 & \cdots &\bm{a}_n &\bm{a}_{n+1} & \cdots &\bm{a}_m
  \end{pmatrix}
  =
    \begin{pmatrix}
    d & 0 & \cdots & 0 & a_{1n+1} & \cdots & a_{1m}\\
    0 & d & \cdots & 0 & a_{2n+1} & \cdots & a_{2m}\\
    \vdots & \vdots & \ddots & \vdots & \vdots & & \vdots\\
    0 & 0 & \cdots & d & a_{nn+1} & \cdots & a_{nm}
  \end{pmatrix}.
\end{equation}

We relate GKZ systems given by the matrix \(A\) to the Gauss-Manin connection of the Fermat deformation.

\begin{proposition}\label{prop:regularity}
  \(M_A(\bm{\beta})\) is regular holonomic for every \(\bm{\beta} \in \C^n\).
\end{proposition}
\begin{proof}
   The row vector \(\begin{pmatrix}1 & \cdots & 1 \end{pmatrix} \in \Q^m\) lies in the \(\Q\)-span of the rows of \(A\), since the sum of those rows is \(\begin{pmatrix}d & \cdots & d\end{pmatrix}\). Hence \(M_A(\bm{\beta})\) is regular holonomic for all \(\bm{\beta} \in \C^n\) by Hotta's theorem \cite[Theorem~2.4.9]{SST}.
 \end{proof}

\begin{proposition}\label{thm:well_definedness}
  For an element \(\overline{\bm{\gamma}} \in H_{n,d}/\Z A\), let \(\bm{\gamma} = (\gamma_1, \ldots, \gamma_n) \in \Z^n\) be a lift of \(\overline{\bm{\gamma}}\) such that \(\gamma_i >0\) for all \(i\).

  \begin{enumerate}[(1)]
  \item   The homomorphism of \(D_m\)-modules
\begin{align*}
D_m &\longrightarrow H^{n}_{\mathrm{Dw}}(Y/S)_{\overline{\bm{\gamma}}} \\ 1 &\longmapsto \omega_{\bm{\gamma}} := [X_1^{\gamma_1-1}\cdots X_n^{\gamma_n-1}dX_1\wedge \cdots \wedge dX_n]
\end{align*} induces a homomorphism
  \(    f_{\bm{\gamma}}: M_A(-\bm{\gamma}) \longrightarrow H^{n}_{\mathrm{Dw}}(Y/S)_{\overline{\bm{\gamma}}}
\)
   of \(D_m\)-modules.
  \item \(f_{\bm{\gamma}}\) is surjective.
  \end{enumerate}
\end{proposition}

\begin{proof}
  Let \(\partial^{\bm{u}} - \partial^{\bm{v}} \in I_A\). By the definition of the \(D_m\)-module structure of \(H^n_{\mathrm{Dw}}(Y/S)\),
  \[
    \partial_{j} \bullet \omega_{\bm{\gamma}} = \left[ \frac{\partial F_A}{\partial \lambda_j}\cdot X^{\bm{\gamma}-\1} dX_1 \wedge \cdots \wedge dX_n \right] = \omega_{\bm{\gamma}+\bm{a}_j}.
  \]
  Since \(A\bm{u} = A\bm{v}\), we have \((\partial^{\bm u} - \partial^{\bm v})\bullet\omega_{\bm{\gamma}} = \omega_{\bm{\gamma} + \sum u_j \bm{a}_j} - \omega_{\bm{\gamma} + \sum v_j \bm{a}_j} = 0\).

  We put \(\eta_i = X_i X^{\bm{\gamma}-\1}dX_1 \wedge \cdots \wedge \widehat{dX_i} \wedge \cdots \wedge dX_n \in \Omega^{n-1}(\mathbb{A}^n/S)^{(d)}\) for every \(i = 1, \ldots, n\). Then
  \begin{align*}
    (d + dF_A \wedge) \eta_i &= (-1)^{i+1}\left(\gamma_iX^{\bm{\gamma}-\1} + X_i\frac{\partial F_A}{\partial X_i}X^{\bm{\gamma}-\1}\right)dX_1 \wedge \cdots \wedge dX_n,
  \end{align*}
  hence
  \begin{equation}\label{eq:euler_relation}
    (E_i + \gamma_i) \bullet \omega_{\bm{\gamma}} = \left[\left(\gamma_iX^{\bm{\gamma}-\1} + X_i\frac{\partial F_A}{\partial X_i}X^{\bm{\gamma}-\1}\right)dX_1 \wedge \cdots \wedge dX_n\right] = 0.
  \end{equation}
  This proves (1).

  \(H^n_{\mathrm{Dw}}(Y/S)_{\overline{\bm{\gamma}}}\) is generated by elements of the form
  \[\omega_{\bm{\beta}} = \left[X^{\bm{\beta}-\1}dX_1 \wedge \cdots \wedge dX_n\right], \quad \bm{\beta} = (\beta_1, \ldots, \beta_n) \text{ is a lift of }\overline{\bm{\gamma}}.\]
For each \(i=1,\ldots,n\), (\ref{eq:euler_relation}) gives
\begin{equation}\label{eq:euler_elevation}
\omega_{\bm{\beta}}
=-\frac{1}{\beta_i}E_i\bullet\omega_{\bm{\beta}}
=-\frac{1}{\beta_i}\sum_{j=1}^m a_{ij}\lambda_j\omega_{\bm{\beta}+\bm{a}_j} =  -\frac{1}{\beta_i}\sum_{j=1}^m a_{ij}\lambda_j\left[X^{\bm{\beta}+\bm{a}_j-\1}dX_1 \wedge \cdots \wedge dX_n\right].
\end{equation}
Every term on the right-hand side has a strictly larger \(i\)-th coordinate
and no smaller coordinate. Iterating \eqref{eq:euler_elevation} for
\(i=1,\ldots,n\), we may therefore express \(\omega_{\bm{\beta}}\) as a linear combination of classes \(\omega_{\bm{\nu}} = \left[X^{\bm{\nu}-\1}dX_1 \wedge \cdots \wedge dX_n\right]\) for which every coordinate of \(\bm{\nu}\) is sufficiently large. By the following lemma, \(\bm{\nu}\) has the form \(\bm{\nu}=\bm{\gamma}+A\bm{c}\) for some \(\bm{c}\in\Z_{\geq 0}^m\). Consequently, \(\omega_{\bm{\nu}}=\partial^{\bm{c}}\bullet\omega_{\bm{\gamma}}\), which proves the surjectivity of \(f_{\bm{\gamma}}\).
\end{proof}

\begin{lemma} \label{lem:conductor}
For each \(\bm{\gamma} = (\gamma_1, \ldots, \gamma_n) \in \Z^n\), there exists an integer \(N > 0\) such that any \(\bm{\nu} = (\nu_1, \ldots, \nu_n) \in \bm{\gamma} + \Z A\) satisfying \(\nu_i \geq N\) for all \(i=1, \dots, n\) belongs to \(\bm{\gamma} + \Z_{\geq 0} A\).
\end{lemma}

\begin{proof}
  Let \(N_0\) be the maximum coordinate among all vectors of the form
  \[\sum_{j=1}^m r_j \bm{a}_j \quad (r_j\in\Z,\ 0 \leq r_j < d),\]
  and put \(N = \max\{1,\ N_0 + \max\{\gamma_1, \ldots,\gamma_n\}\}\).
  Given \(\bm{\nu} = (\nu_1, \ldots, \nu_n) \in\bm{\gamma}+\Z A\), choose \(\bm{b}=(b_1, \ldots, b_m)\in\Z^m\) such that
  \(\bm{\nu}-\bm{\gamma}=A\bm{b}\).
  For each \(j\), let \(r_j\) be the residue of \(b_j\) modulo \(d\),
  with \(0\leq r_j<d\). Then
  \[
    \bm{\nu} - \bm{\gamma} - \sum_{j=1}^m r_j\bm{a}_j = \sum_{j=1}^m(b_j-r_j)\bm{a}_j\in d\Z^n.
  \]
  Hence there exist integers \(q_1,\ldots, q_n\) such that
  \[
    \bm{\nu} - \bm{\gamma} = \sum_{j=1}^m r_j\bm{a}_j + \sum_{i=1}^n q_i d\bm{e}_i.
  \]
  If \(\nu_i\geq N\) for all \(i\), we have
  \[
    dq_i = \nu_i - \gamma_i - \sum_{j=1}^mr_ja_{ij} \geq N - \max\{\gamma_1,\ldots,\gamma_n\} - N_0 \geq 0
  \]
   for each \(i=1,\ldots,n\). Thus \(q_i\geq0\) for all \(i\). Since \(d\bm{e}_i=\bm{a}_i\),
  \[
    \bm{\nu} - \bm{\gamma} = \sum_{j=1}^m r_j\bm{a}_j+\sum_{i=1}^n q_i\bm{a}_i \in \Z_{\geq 0}A.
  \]
  Therefore \(\bm{\nu} \in \bm{\gamma} + \Z_{\geq 0}A\), as required.
\end{proof}

Since \(\Z A \neq \Z^n\), we normalize \(A\) in order to apply the results from the latter part of the previous section. Let \(PAQ = \begin{pmatrix} D & O \end{pmatrix}\) be the Smith normal form of \(A\), where \(P\in \mathrm{GL}_n(\Z), Q \in \mathrm{GL}_m(\Z)\), and
\[
  \begin{pmatrix}D & O\end{pmatrix}=
  \begin{pmatrix}
    \alpha_1 & 0 & \cdots & 0 & 0 & \cdots & 0\\
    0 & \alpha_2 & \cdots & 0 & 0 & \cdots & 0\\
    \vdots & \vdots & \ddots & \vdots & \vdots & & \vdots\\
    0 & 0 & \cdots & \alpha_n & 0 & \cdots & 0
  \end{pmatrix} \quad (\alpha_1,\ldots, \alpha_n \in \Z \setminus \{0\}).
\]
We define
\begin{equation}\label{nomalization}
T = D^{-1}P \in \mathrm{GL}_n(\Q) \text{ and } \widetilde{A} = TA.
\end{equation}
Then \(\widetilde{A}\) is an integer matrix and \(\Z\widetilde{A} = \Z^n\). Clearly \(A\) and \(\widetilde{A}\) define the same toric ideal, \(M_A(\bm{\beta})\) coincides with \(M_{\widetilde{A}}(T\bm{\beta})\) for any \(\bm{\beta} \in \C^n\).

\begin{proposition}
  \(\widetilde{A}\) is positive and homogeneous.
\end{proposition}

\begin{proof}
For a semigroup \(G\) in \(\Z^n\), let \(TG\) be the image of \(G\) by \(T\). Then
\begin{align*}
  \Z_{\geq 0}\widetilde A\cap(-\Z_{\geq 0}\widetilde A)
  &= T(\Z_{\geq 0} A)\cap T(-\Z_{\geq 0} A)\\
  &= T\left(\Z_{\geq 0} A\cap(-\Z_{\geq 0} A)\right)\\
  &= \{0\}.
\end{align*}

  Write \(\widetilde{A} = \begin{pmatrix}\widetilde{\bm{a}}_1 & \cdots & \widetilde{\bm{a}}_m\end{pmatrix}\). Since \(\Z \widetilde{A} = \Z^n\), every \(\bm{x} \in \Z^n\) can be written as a \(\Z\)-linear combination of the vectors \(\widetilde{\bm{a}}_i\). Define a group homomorphism by
\[
  h:\Z^n \longrightarrow \Q; \bm{x} \longmapsto \frac{1}{d}\begin{pmatrix}1 & \cdots & 1\end{pmatrix}T^{-1}\bm{x}.
\]
By the definition of \(A\), \(\begin{pmatrix}1 & \cdots & 1\end{pmatrix}\bm{a}_i = d\) for all \(i=1, \ldots, m\). Therefore, \(h\) takes integral values and \(h(\widetilde{\bm{a}}_i) = 1\) for all \(i = 1, \ldots, m\).
\end{proof}

\begin{lemma}\label{lemma:non-resonance}
  For any \(\bm{\beta} \in \Q^n\), \(T(\bm{\beta} + kd\1) \notin \mathrm{sRes}(\widetilde{A})\) and \(T(\bm{\beta} - kd\1) \notin \DD\mathrm{sRes}(\widetilde{A})\) for sufficiently large integer \(k \gg 0\).
\end{lemma}

\begin{proof}
  By Proposition~\ref{thm:sRes}, there exists \(\bm{\delta}_{\widetilde{A}} \in \Z_{\geq 0}\widetilde{A}\) such that \((\R_{\geq 0}\widetilde{A} + \bm{\delta}_{\widetilde{A}}) \cap \mathrm{sRes}(\widetilde{A}) = \emptyset\). Since \(\R_{\geq 0}A = \R_{\geq 0}^n\), we have \(\bm{\beta} + kd\1 - T^{-1}\bm{\delta}_{\widetilde{A}}\in \R_{\geq 0}A\) for sufficiently large \(k\). Then \(T(\bm{\beta} + kd\1) \in \bm{\delta}_{\widetilde{A}} + \R_{\geq 0}\widetilde{A}\), and thus \(T(\bm{\beta} + kd\1) \notin \mathrm{sRes}(\widetilde{A})\).

  Similarly, \(\bm{\beta} - kd\1 \in -(\R_{\geq 0}A)^\circ\) for sufficiently large \(k\). \(T\) maps \(-(\R_{\geq 0}A)^\circ\) to \(-(\R_{\geq 0}\widetilde{A})^\circ\). Hence \(T(\bm{\beta} - kd\1) \notin \DD\mathrm{sRes}(\widetilde{A})\) by Proposition~\ref{thm:DsRes}.
\end{proof}

\begin{proposition}\label{prop:contiguity-image}
  For any \(\bm{\beta} \in \Q^n\),
  \[
    \FT(\phi^{\widetilde{A}}_{!+}\mathcal{K}_{T\bm{\beta}-\1}) \cong \mathrm{Im}(M_A(-\bm{\beta} - kd\1) \xlongrightarrow{\cdot (\partial_1\cdots \partial_n)^{2k}} M_A(-\bm{\beta} + kd\1))
  \]
  for sufficiently large integer \(k\).
\end{proposition}
\begin{proof}
  Since \(\bm{a}_i=d\bm{e}_i\) for \(i=1,\ldots,n\), the vector
  \(T(d\1)=\widetilde{\bm{a}}_1+\cdots+\widetilde{\bm{a}}_n\)
  is integral. Thus the Kummer modules with parameters
  \(T\bm{\beta}\pm kT(d\1)-\1\)
  are isomorphic to \(\mathcal K_{T\bm{\beta}-\1}\).
  For all sufficiently large \(k\),
  \[
    M_{\widetilde{A}}(T(-\bm{\beta} - kd\1)) \cong \FT(\phi^{\widetilde{A}}_!\mathcal{K}_{T\bm{\beta}-\1}) \longrightarrow \FT(\phi^{\widetilde{A}}_+\mathcal{K}_{T\bm{\beta}-\1}) \cong M_{\widetilde{A}}(T(-\bm{\beta} + kd\1))
  \]
  by Propositions \ref{thm:GKZ=FT+} and \ref{thm:GKZ=FT!} and the preceding lemma. Moreover,
  \[
    T(-\bm{\beta} + kd\1) - T(-\bm{\beta} - kd\1) = T(2kd\1) = 2k(\widetilde{\bm{a}}_1+ \cdots +\widetilde{\bm{a}}_n).
  \]
  Hence the homomorphism \(M_{\widetilde{A}}(T(-\bm{\beta} - kd\1)) \rightarrow M_{\widetilde{A}}(T(-\bm{\beta} + kd\1))\) coincides, up to a nonzero scalar, with the contiguity operator \(\cdot (\partial_1\cdots\partial_n)^{2k}\) \cite[Proposition~1.15]{reichelt}.
\end{proof}

\begin{proposition}\label{isomorphism_on_Gm}
  If \(H^{n}_{\mathrm{Dw}}(Y/S)_{\overline{\bm{\gamma}}} \neq 0\), then
  \[
    \FT(f_{\bm{\gamma}}) : \FT(M_A(-\bm{\gamma})) \rightarrow \FT(H^{n}_{\mathrm{Dw}}(Y/S)_{\overline{\bm{\gamma}}})
  \]
  is an isomorphism on \(\mathbb{G}_m^m \subset \mathbb{A}^m_\C\).
\end{proposition}

\begin{proof}
 We use the same notation \(I_A\) for the toric ideal in \(\C[\lambda_1,\ldots,\lambda_m]\) obtained by replacing each \(\partial_i\) with \(\lambda_i\).
 Let \(O_{\widetilde A}:=\phi^{\widetilde A}(\mathbb{G}_m^n)\). The Zariski closure \(\overline{O_{\widetilde{A}}}\) of \(O_{\widetilde{A}}\) in \(\mathbb{A}^m_\C\) is defined by the toric ideal \(I_{\widetilde{A}} = I_A\). The generator \(1\) of \(\FT(M_A(-\bm{\gamma}))\) is annihilated by \(I_A\). Thus its support is contained in \(\overline{O_{\widetilde{A}}}\). By the surjectivity of \(\FT(f_{\bm{\gamma}})\), the support of \(\FT(H^{n}_{\mathrm{Dw}}(Y/S)_{\overline{\bm{\gamma}}})\) is also contained in \(\overline{O_{\widetilde{A}}}\).

  Let \(\iota:\mathbb{G}_m^n \cong O_{\widetilde{A}} = \overline{O_{\widetilde{A}}} \cap \mathbb{G}_m^m \hookrightarrow \mathbb{G}_m^m\) be the closed immersion. Then, by \cite[Proposition~2.1]{SW09},
  \[
  H^{m-n}\iota^!\FT(H^{n}_{\mathrm{Dw}}(Y/S)_{\overline{\bm{\gamma}}})|_{\mathbb{G}_m^m} \leftarrow H^{m-n}\iota^!\FT(M_A(-\bm{\gamma}))|_{\mathbb{G}_m^m} \cong H^{m-n}\iota^!\iota_+\mathcal{K}_{T\bm{\gamma}-\1} \cong \mathcal{K}_{T\bm{\gamma}-\1}
  \]
  is injective because it is a nonzero homomorphism from an irreducible \(\D_{\mathbb{G}_m^n}\)-module \(\mathcal{K}_{T\bm{\gamma}-\1}\). By Kashiwara's equivalence \cite[Proposition~7.4.1]{mebkhout}, \(\FT(f_{\bm{\gamma}})\) is therefore injective on \(\mathbb{G}_m^m\).
\end{proof}

\subsection{Main Theorem}
Let \(A\) be the matrix defined in (\ref{eq:matrix}). Recall that \(U\subset S=\mathbb{A}^m_\C\) is the open subset over which the family \(Y/S\) in (\ref{eq:fermat-deformation}) is smooth, and put \(Y_U=Y\times_S U\). Let \(\widetilde{A}=TA\) be the normalization in (\ref{nomalization}). For any two lifts \(\bm{\gamma}, \bm{\gamma}'\in\Z^n\) of \(\overline{\bm{\gamma}}\in H_{n,d}/\Z A\), we have \(T(\bm{\gamma}-\bm{\gamma}')\in \Z\widetilde{A}\). Since Kummer modules whose parameters differ by an integral vector are isomorphic, \(\mathcal{K}_{T\bm{\gamma}-\1}\cong\mathcal{K}_{T\bm{\gamma}'-\1}\). Hence the isomorphism class of \(\FT(\phi^{\widetilde{A}}_{!+}\mathcal{K}_{T\bm{\gamma}-\1})\) depends only on \(\overline{\bm{\gamma}}\).

\begin{theorem}
  \label{thm:gauss-manin-gkz}
  If \(H^{n-2}_{\mathrm{dR}}(Y_U/U)^{\mathrm{prim}}_{\overline{\bm{\gamma}}} \neq 0\), then
  \[
    H^{n-2}_{\mathrm{dR}}(Y_U/U)^{\mathrm{prim}}_{\overline{\bm{\gamma}}} \cong \FT(\phi^{\widetilde{A}}_{!+}\mathcal{K}_{T\bm{\gamma}-\1})|_U
  \]
  for any lift \(\bm{\gamma}\in \Z^n\) of \(\overline{\bm{\gamma}}\).
    \end{theorem}
    \begin{proof}
      Let \(j:\mathbb{G}_m^m \hookrightarrow \mathbb{A}^m_\C\) be the inclusion. The monomial map \(\phi^{\widetilde{A}}\) factors as
      \[
        \mathbb{G}_m^n \xlongrightarrow{\iota} \mathbb{G}_m^m \xlongrightarrow{j} \mathbb{A}^m_\C.
      \]
Here \(\iota\) is given by
\(\iota(t_1,\ldots,t_n) = (t^{\widetilde{\bm{a}}_1},\ldots,t^{\widetilde{\bm{a}}_m})\). Since \(\Z\widetilde{A}=\Z^n\), the homomorphism of Laurent polynomial rings induced by \(\iota\) is surjective, so \(\iota\) is a closed immersion.

From Proposition \ref{thm:well_definedness} and Lemma \ref{lemma:non-resonance}, we have a surjective homomorphism
\[
  \Phi:\phi^{\widetilde{A}}_!\mathcal{K}_{T\bm{\gamma}-\1} = j_!\iota_+\mathcal{K}_{T\bm{\gamma}-\1} \twoheadrightarrow \FT(H^{n}_{\mathrm{Dw}}(Y/S)_{\overline{\bm{\gamma}}}).
\]
Since the inverse image by the open immersion \(j\) is exact, the sequence
\[
  0 \rightarrow (\ker \Phi)|_{\mathbb{G}_m^m} \rightarrow \iota_+\mathcal{K}_{T\bm{\gamma}-\1} \rightarrow \FT(H^{n}_{\mathrm{Dw}}(Y/S)_{\overline{\bm{\gamma}}})|_{\mathbb{G}_m^m}
\]
is exact, and we have \((\ker \Phi)|_{\mathbb{G}_m^m} = 0\) by Proposition \ref{isomorphism_on_Gm}. By adjunction,
\[
  \operatorname{Hom}_{D^b(\D_{\mathbb{A}^m_\C})}(\ker \Phi,j_+\iota_+\mathcal K_{T\bm{\gamma}-\1})
  \cong \operatorname{Hom}_{D^b(\D_{\mathbb{G}_m^m})}(j^!\ker \Phi,\iota_+\mathcal K_{T\bm{\gamma}-\1})
  =0.
\]
This implies that the composition \(\ker \Phi \rightarrow \phi^{\widetilde{A}}_!\mathcal{K}_{T\bm{\gamma}-\1} \rightarrow \phi^{\widetilde{A}}_+\mathcal{K}_{T\bm{\gamma}-\1}\) is zero. Hence there exists a unique morphism
\[
  \Psi:\FT(H^{n}_{\mathrm{Dw}}(Y/S)_{\overline{\bm{\gamma}}}) \longrightarrow \phi^{\widetilde{A}}_{!+}\mathcal{K}_{T\bm{\gamma}-\1} = \mathrm{Im}(\phi^{\widetilde{A}}_!\mathcal{K}_{T\bm{\gamma}-\1} \rightarrow \phi^{\widetilde{A}}_+\mathcal{K}_{T\bm{\gamma}-\1})
\]
such that the diagram
\[
\begin{tikzcd}
\phi^{\widetilde{A}}_!\mathcal{K}_{T\bm{\gamma}-\1} \arrow[r, two heads, "\Phi"]
      \arrow[dr, two heads]
&
\FT(H^{n}_{\mathrm{Dw}}(Y/S)_{\overline{\bm{\gamma}}}) \arrow[d, dashed, "\Psi"]\\
&
\phi^{\widetilde{A}}_{!+}\mathcal{K}_{T\bm{\gamma}-\1}
\end{tikzcd}
\]
commutes. Thus we obtain a surjection
\[
  \FT(\Psi):H^{n}_{\mathrm{Dw}}(Y/S)_{\overline{\bm{\gamma}}} \twoheadrightarrow \FT(\phi^{\widetilde{A}}_{!+}\mathcal{K}_{T\bm{\gamma}-\1}).
\]

Then it is enough to prove the inequality of the holonomic ranks
\begin{equation}
  \label{eq:inequality}
  \rank H^{n}_{\mathrm{Dw}}(Y/S)_{\overline{\bm{\gamma}}} \leq \rank \FT(\phi^{\widetilde{A}}_{!+}\mathcal{K}_{T\bm{\gamma}-\1}).
\end{equation}
By Proposition \ref{prop:contiguity-image},
\begin{align*}
  \FT(\phi^{\widetilde{A}}_{!+}\mathcal{K}_{T\bm{\gamma}-\1}) &= \mathrm{Im}(M_A(-\bm{\gamma} - kd\1) \xlongrightarrow{\cdot (\partial_1\cdots \partial_n)^{2k}} M_A(-\bm{\gamma} + kd\1)).
  \end{align*}
  The cokernel of this morphism is
  \[
    C_k = D_m \big/ \left(\sum_{i=1}^nD_m(E_i + \gamma_i-kd)+D_m I_A + D_m(\partial_1\cdots \partial_n)^{2k}\right),
  \]
and
\[
  \rank C_k = \rank M_A(-\bm{\gamma} + kd\1) - \rank \FT(\phi^{\widetilde{A}}_{!+}\mathcal{K}_{T\bm{\gamma}-\1}).
\]

Thus to show the inequality (\ref{eq:inequality}), it is enough to show

\begin{equation}
  \label{eq:main_inequality}
  \rank C_k \leq \rank M_A(-\bm{\gamma} + kd\1) - \rank H^{n}_{\mathrm{Dw}}(Y/S)_{\overline{\bm{\gamma}}}
\end{equation}
for a sufficiently large \(k\). This inequality is proved in Proposition~\ref{prop:inequality} below.
\end{proof}

  Since the minimal extension of an irreducible \(\D\)-module is irreducible, we obtain the following corollary.
  \begin{corollary}
    Each nonzero component \(H^{n-2}_{\dR}(Y_U/U)^{\mathrm{prim}}_{\overline{\bm{\gamma}}}\) is an irreducible \(\D_U\)-module.
  \end{corollary}

  \section{An upper bound for the holonomic rank}
  Throughout this section, let \(A\) be the matrix defined in (\ref{eq:matrix}).
    \subsection{Initial ideals}
  For \(P_1,\ldots,P_s\) in \(\C[\partial_1,\ldots,\partial_m]\), we write \(\langle P_1,\ldots,P_s\rangle\) for the ideal generated by \(P_1,\ldots,P_s\).
  For two elements \(\bm{v}, \bm{w} \in \R^m\), \(\bm{v}\cdot \bm{w}\) denotes the inner product.
\begin{definition}
  Let \(\bm{w} = (w_1, \ldots, w_m) \in \R^m\) and \(f = \sum_{\bm{u}} c_{\bm{u}} \partial^{\bm{u}} \in \C[\partial_1, \ldots, \partial_m]\). The initial form \(\mathrm{in}_{\bm{w}}(f)\) is the sum of all terms \(c_{\bm{u}} \partial^{\bm{u}}\) for which \(\bm{w} \cdot \bm{u}\) is maximal. The ideal
  \[
    \mathrm{in}_{\bm{w}}(I_A) = \langle \mathrm{in}_{\bm{w}}(f) \mid f \in I_A \rangle.
  \]
  of \(\C[\partial_1, \ldots, \partial_m]\) is called the initial ideal of \(I_A\).

  For \(P=\sum_{\bm{u},\bm{v}}c_{\bm{u},\bm{v}}\lambda^{\bm{u}}\partial^{\bm{v}}\in D_m\),
  with \(\bm{u},\bm{v}\in\Z_{\geq0}^m\), the initial form
  \(\mathrm{in}_{(-\bm{w},\bm{w})}(P)\) is the sum of all nonzero terms
  for which \(-\bm{w}\cdot\bm{u}+\bm{w}\cdot\bm{v}\) is maximal.
  We set \(\mathrm{in}_{(-\bm{w},\bm{w})}(0)=0\).
  For a left ideal \(I\subset D_m\), we define
  \[
    \mathrm{in}_{(-\bm{w},\bm{w})}(I) := \sum_{P\in I}D_m\,\mathrm{in}_{(-\bm{w},\bm{w})}(P) \subset D_m.
  \]
\end{definition}

\begin{lemma}\label{generic_weight_vector}
  Let \(w_{n+1},\ldots,w_m\) be positive real numbers which are
  linearly independent over \(\Q\), and set \(\bm{w}=(0,\ldots,0,w_{n+1},\ldots,w_m)\).
  \begin{enumerate}[(1)]
  \item If \(\bm{u} = (u_1, \ldots, u_m),\bm{v} = (v_1, \ldots, v_m) \in \Z_{\geq 0}^m\) satisfy \(A \bm{u}=A \bm{v}\) and \(\bm{w} \cdot \bm{u} = \bm{w} \cdot \bm{v}\), then \(\bm{u}=\bm{v}\).

  \item \(\mathrm{in}_{\bm{w}}(I_A) = \langle \mathrm{in}_{\bm{w}}(\partial^{\bm{u}}-\partial^{\bm{v}}) \mid A\bm{u} = A\bm{v} \rangle\).
  \item \(\mathrm{in}_{\bm{w}}(I_A)\) is a monomial ideal.
  \end{enumerate}
\end{lemma}

\begin{proof}
Suppose that \(A\bm{u}=A\bm{v}\) and
\(\bm{w}\cdot\bm{u}=\bm{w}\cdot\bm{v}\). Then
\[
  \sum_{j=n+1}^m w_j(u_j-v_j)=0.
\]
By the linear independence of \(w_{n+1},\ldots,w_m\), \(u_j=v_j\) for every \(j=n+1, \ldots, m\). 
Hence
\[
  0 = A(\bm{u}-\bm{v}) = d\sum_{i=1}^n(u_i-v_i)\bm{e}_i =
  \begin{pmatrix}
    d(u_1-v_1) \\ \vdots \\ d(u_n - v_n)
  \end{pmatrix}
  ,
\]
so \(u_i=v_i\) also for \(i\leq n\). This proves (1).

Let
\(f\) be an element of \(I_A\). If \(\partial^{\bm{u}}\) appears in \(\operatorname{in}_{\bm{w}}(f)\), then there exists another term \(\partial^{\bm{v}}\) in \(f\) such that \(A(\bm{u} - \bm{v}) = 0\). By (1), \(\bm{w} \cdot \bm{u} > \bm{w} \cdot \bm{v}\), and hence \(\operatorname{in}_{\bm{w}}(\partial^{\bm{u}} - \partial^{\bm{v}}) = \partial^{\bm{u}}\).

Finally, (3) follows from (1) and (2).
\end{proof}

\subsection{Fake Indicial Ideals and Fake Exponents}
We use the notation of \S \ref{sec:definition_gkz}. Throughout this subsection, let \(\bm{w}\) be a vector as in Lemma~\ref{generic_weight_vector}. The results from \cite[\S\S~4.1 and 4.6]{SST} cited below assume that \(\bm{w}\) is generic. This assumption ensures that \(\operatorname{in}_{\bm{w}}(I_A)\) is a monomial ideal and is satisfied when \(\bm{w}\) is chosen as in Lemma~\ref{generic_weight_vector}.

Write \(\theta_i=\lambda_i\partial_i\). Since \(\theta_i, \theta_j\) commute with each other, \(\C[\theta_1, \ldots, \theta_m]\) is a commutative subalgebra of \(D_m\). For \(r\in\Z_{\geq0}\), we set
\[
  [\theta_j]_r:=\prod_{\ell=0}^{r-1}(\theta_j-\ell),
  \qquad [\theta_j]_0:=1.
\]
For \(\bm{u} = (u_1, \ldots, u_m) \in \Z_{\geq0}^m\), write
\[
  [\theta]_{\bm{u}}:=\prod_{i=1}^m[\theta_i]_{u_i}.
\]
We use the same notation \(\langle P_1,\ldots,P_s\rangle\) as for ideals in \(\C[\partial_1,\ldots,\partial_m]\) to denote the ideal generated by \(P_1,\ldots,P_s\) in \(\C[\theta_1,\ldots,\theta_m]\).
For \(\bm{\beta}\in\C^n\), we set
\[
  \langle A\bm{\theta}-\bm{\beta}\rangle
  :=\left\langle \sum_{j=1}^m a_{ij}\theta_j-\beta_i
  \;\middle|\; i=1,\ldots,n\right\rangle
  \subset\C[\theta_1,\ldots,\theta_m].
\]

\begin{definition}
  The ideal
  \[
    \widetilde{\mathrm{in}}_{\bm{w}}(I_A)
    :=\left\langle [\theta]_{\bm{u}}
      \;\middle|\; \partial^{\bm{u}}\in\mathrm{in}_{\bm{w}}(I_A), \bm{u} \in \Z_{\geq 0}^m \right \rangle
    \subset\C[\theta_1, \ldots, \theta_m].
  \]
  is called the distraction of \(\mathrm{in}_{\bm{w}}(I_A)\).
  The fake indicial ideal of the GKZ system is
  \[
    \fin_{\bm{w}}(H_A(\bm{\beta}))
    :=
    \widetilde{\mathrm{in}}_{\bm{w}}(I_A)
    +
    \langle A\bm{\theta}-\bm{\beta}\rangle \subset \C[\theta_1, \ldots, \theta_m].
  \]
\end{definition}

\begin{definition}[{\cite[\S~3.2, p.~109]{SST}}]\label{def:standard_pair}
  A standard pair of \(\mathrm{in}_{\bm{w}}(I_A)\) is a pair \((\partial^{\bm{u}}, \sigma)\) where \(\bm{u} \in \Z_{\geq 0}^m\) and \(\sigma \subset \{1, \ldots, m\}\) satisfying the following conditions:
  \begin{enumerate}[(i)]
  \item \(u_j = 0\) for all \(j \in \sigma\),
  \item \(\partial^{\bm{u}}\cdot \prod_{j \in \sigma}\partial_j^{b_j} \notin \mathrm{in}_{\bm{w}}(I_A)\) for any \((b_j)_{j \in \sigma} \in \Z_{\geq 0}^{\#\sigma}\),
  \item For every \(\ell\in\{1,\ldots,m\}\setminus\sigma\), there exist
  \((b_j)_{j\in\sigma\cup\{\ell\}}\in\Z_{\geq0}^{|\sigma|+1}\) such that
  \[
    \partial^{\bm{u}}\partial_\ell^{b_\ell}\prod_{j\in\sigma}\partial_j^{b_j}\in\mathrm{in}_{\bm{w}}(I_A).
  \]
  \end{enumerate}

  A standard pair \((\partial^{\bm{u}},\sigma)\) is called top-dimensional if
  \(\#\sigma = n = \operatorname{rank}(A)\), and it is called embedded if \(\#\sigma < n\).
\end{definition}

We denote the set of standard pairs of \(\mathrm{in}_{\bm{w}}(I_A)\) by \(\mathcal{S}(\mathrm{in}_{\bm{w}}(I_A))\) and the set of top-dimensional standard pairs of \(\mathrm{in}_{\bm{w}}(I_A)\) by \(\mathcal{T}(\mathrm{in}_{\bm{w}}(I_A))\). We write \([n]=\{1,\ldots,n\}\).

\begin{proposition}
  Every standard pair \((\partial^{\bm{u}}, \sigma) \in \mathcal{S}(\mathrm{in}_{\bm{w}}(I_A))\) satisfies \(\sigma \subset [n]\). In particular, every top-dimensional standard pair is of the form \((\partial^{\bm{u}}, [n])\).
\end{proposition}
\begin{proof}
  Let \((\partial^{\bm{u}}, \sigma)\) be a standard pair of \(\mathrm{in}_{\bm{w}}(I_A)\). \(\partial_j^d - \partial_1^{a_{1j}} \cdots \partial_n^{a_{nj}} \in I_A\) for any \(j \in \{n+1, \ldots, m\}\). \(\partial_j^d \in \mathrm{in}_{\bm{w}}(I_A)\) by the choice of \(\bm{w}\), and hence \(\partial^{\bm{u}}\cdot\partial^{d}_j \in \mathrm{in}_{\bm{w}}(I_A)\).
To satisfy condition~(ii) of Definition~\ref{def:standard_pair}, we must have \(\sigma \subset [n]\).
\end{proof}

\begin{proposition}[{\cite[Lemma~4.1.3]{SST}}]
  \(\C[\theta_1, \ldots, \theta_m]/\fin_{\bm{w}}(H_A(\bm{\beta}))\) is Artinian and
  \[
    \sqrt{\fin_{\bm{w}}(H_A(\bm{\beta}))} = \bigcap_{(\partial^{\bm{u}}, \sigma) \in \mathcal{S}(\mathrm{in}_{\bm{w}}(I_A))} (\langle \theta_i - u_i \mid i \notin \sigma \rangle + \langle A\bm{\theta} - \bm{\beta} \rangle).
  \]
\end{proposition}

Therefore the zeros of \(\fin_{\bm{w}}(H_A(\bm{\beta}))\) correspond to vectors \(\bm{p} \in \C^m\) such that
\begin{align}
   &A \bm{p} = \bm{\beta}, \label{eq:fake_exponent1}\\
   &p_j = u_j \text{ for all }j \notin \sigma. \label{eq:fake_exponent2}
\end{align}

Given a standard pair \((\partial^{\bm{u}}, \sigma)\), there exists at most one vector \(\bm{p} \in \C^m\) satisfying \eqref{eq:fake_exponent1} and \eqref{eq:fake_exponent2}. If such a vector exists, we say that the standard pair is \emph{associated with} \(\bm{\beta}\) and denote the vector by \(\bm{\beta}^{(\partial^{\bm{u}}, \sigma)}\).

\begin{proposition}\label{prop:w-flat}
The following statements hold.
\begin{enumerate}[(1)]
\item For all sufficiently large \(k\), no embedded standard pair of \(\mathrm{in}_{\bm{w}}(I_A)\) is associated with \(-\bm{\gamma} + kd\1\).

\item For every \(k\), the map
\begin{align*}
\mathcal{T}(\mathrm{in}_{\bm{w}}(I_A)) &\longrightarrow \C^m \\
(\partial^{\bm{u}},[n]) &\longmapsto (-\bm{\gamma} + kd\1)^{(\partial^{\bm{u}},[n])}
\end{align*}
is injective.
\end{enumerate}
\end{proposition}

\begin{proof}
  (1) Let \((\partial^{\bm{u}},\sigma)\) be an embedded standard pair. Since \(\sigma \subsetneq [n]\), we may choose \(i\in[n]\setminus\sigma\). If \((\partial^{\bm{u}},\sigma)\) is associated with \(-\bm{\gamma} + kd\1\), then
  \[
    -\gamma_i + kd = \sum_{j=1}^m a_{ij}u_j.
  \]
  Thus if we choose an integer \(k\) such that \(k > (\gamma_i + \sum_{j=1}^m a_{ij}u_j)/d\) for every embedded standard pair \((\partial^{\bm{u}},\sigma)\) and every \(i\in[n]\setminus\sigma\), it follows that no embedded standard pair is associated with \(-\bm{\gamma} + kd\1\).

  (2) Let
  \((\partial^{\bm{u}},[n]), (\partial^{\bm{v}},[n])\) be two top-dimensional standard pairs, and suppose that
  \[
    (-\bm{\gamma} + kd\1)^{(\partial^{\bm{u}},[n])} = (-\bm{\gamma} + kd\1)^{(\partial^{\bm{v}},[n])}.
  \]

  From Definition \ref{def:standard_pair} and (\ref{eq:fake_exponent2}),  we obtain \(\bm{u} = \bm{v}\).
\end{proof}

\begin{lemma}\label{T=vol}
The map
\begin{align*}
\Phi : \mathcal{T}(\mathrm{in}_{\bm{w}}(I_A)) &\longrightarrow \Z A/d\Z^n\\
(\partial^{\bm{u}}, [n]) &\longmapsto \overline{A\bm{u}}
\end{align*}
is a bijection.
\end{lemma}

\begin{proof}
Let \((\partial^{\bm{u}}, [n]), (\partial^{\bm{v}}, [n]) \in \mathcal{T}(\mathrm{in}_{\bm{w}}(I_A))\), and suppose that \(\Phi((\partial^{\bm{u}}, [n])) = \Phi((\partial^{\bm{v}}, [n]))\).
Then there exists \(\bm{t} \in \Z^n\) such that \(A(\bm{u}-\bm{v})=d\bm{t}\). 
Write \(\bm{t}\) as
\[
  \bm{t} = \bm{t}^+ - \bm{t}^- \qquad (\bm{t}^+=(t^+_1, \ldots, t^+_n), \bm{t}^- = (t^-_1, \ldots, t^-_n) \in \Z_{\geq 0}^n).
\]
Embed \(\bm{t}^\pm\) into \(\Z_{\geq 0}^m\) by setting \(\hat{\bm{t}}^\pm := (t^\pm_1,\dots,t^\pm_n, 0,\ldots,0)\).
Then
\begin{equation}\label{eq:binomial}
  \partial^{\bm{u} + \hat{\bm{t}}^-} - \partial^{\bm{v} + \hat{\bm{t}}^+} \in I_A.
\end{equation}
since
\[
  A((\bm{u}+\widehat{\bm{t}}^-) - (\bm{v}+\widehat{\bm{t}}^+)) = A(\bm{u} - \bm{v}) - d\bm{t} = 0.
\]
On the other hand, \(\partial^{\bm{u}+\widehat{\bm{t}}^-}\) and \(\partial^{\bm{v}+\widehat{\bm{t}}^+}\) are not contained in \(\mathrm{in}_{\bm{w}}(I_A)\) by the definition of standard pairs.

Suppose that \(\partial^{\bm{u}+\widehat{\bm{t}}^-} \neq \partial^{\bm{v}+\widehat{\bm{t}}^+}\). Then one of the two monomials in (\ref{eq:binomial}) must belong to \(\mathrm{in}_w(I_A)\), contradicting the two non-membership statements above. It follows that \(\partial^{\bm{u}+\widehat{\bm{t}}^-} = \partial^{\bm{v}+\widehat{\bm{t}}^+}\), and therefore \(\bm{u}+\widehat{\bm{t}}^- = \bm{v}+\widehat{\bm{t}}^+\). Since the last \(m-n\) coordinates of \(\widehat{\bm{t}}^\pm\) and the first \(n\) coordinates of \(\bm{u}\) and \(\bm{v}\) are zero, we obtain \(u_j=v_j\) for \(j = 1,\dots,m\). Thus \(\bm{u} = \bm{v}\).

We next prove surjectivity. Fix \({\bm{\beta}} = ({\beta}_1, \ldots, {\beta}_n) \in \Z A\), and choose \(\bm{u} \in \Z_{\geq 0}^m\) in the set
\begin{equation}
  \label{eq:the_set}
  \left\{\bm{u}=(u_1, \ldots, u_m) \in \Z_{\geq 0}^m \mid u_j = 0 \text{ for } j \in [n] \text{ and } A\bm{u} \equiv \bm{\beta} \bmod{d\Z^n}\right\}
\end{equation}
whose \(\bm{w}\)-weight is minimal. We claim that \((\partial^{\bm{u}},[n])\) is a top-dimensional standard pair of \(\mathrm{in}_{\bm{w}}(I_A)\).

For any \(\bm{t} \in \Z_{\geq 0}^{m}\) such that \(t_j = 0\) for every \(j > n\), we have \(\partial^{\bm{u}+\bm{t}} \notin \mathrm{in}_{\bm{w}}(I_A)\).
Indeed, if \(\partial^{\bm{u}+\bm{t}} \in \mathrm{in}_{\bm{w}}(I_A)\), Lemma~\ref{generic_weight_vector}(2) gives \(\bm{s} = (s_1, \ldots, s_m) \in \Z_{\geq 0}^m\) such that
\[
  \partial^{\bm{u}+\bm{t}}-\partial^{\bm{s}} \in I_A,
  \qquad
  \mathrm{in}_{\bm{w}}(\partial^{\bm{u}+\bm{t}}-\partial^{\bm{s}})
  = \partial^{\bm{u}+\bm{t}}.
\]
Then
\[
  A\bm{s}=A(\bm{u}+\bm{t}) \equiv A\bm{u} \pmod{d\Z^n}, \quad \bm{w}\cdot\bm{s} < \bm{w}\cdot(\bm{u}+\bm{t}) = \bm{w}\cdot\bm{u}.
\]
If we set
\(\hat{\bm{s}} = (0, \ldots, 0, s_{n+1}, \ldots, s_m)\) then \(\hat{\bm{s}}\) belongs to the set (\ref{eq:the_set}) and \(\bm{w}\cdot\hat{\bm{s}} = \bm{w}\cdot{\bm{s}} < \bm{w}\cdot\bm{u}\), contradicting the choice of \(\bm{u}\).

Finally, for every \(j>n\), we have \(\partial_j^d\in \mathrm{in}_{\bm{w}}(I_A)\), and hence \(\partial^{\bm{u}}\partial_j^d\in \mathrm{in}_{\bm{w}}(I_A)\).
Thus \((\partial^{\bm{u}},[n])\) is a top-dimensional standard pair and \(\Phi(\partial^{\bm{u}},[n]) = \overline{\bm{\beta}}\).
\end{proof}

Let \(d\mu\) be the Lebesgue measure on \(\R^n\) normalized so that the standard simplex has volume one. We define the normalized volume of \(A\) by
\[
  \operatorname{vol}(A) := \frac{1}{[\Z^n:\Z A]} \int_{\operatorname{conv}(0,\bm{a}_1,\ldots,\bm{a}_m)}d\mu,
\]
where \(\operatorname{conv}(0,\bm{a}_1,\ldots,\bm{a}_m))\) denotes the convex hull of \(\bm{a}_1,\ldots,\bm{a}_m\) and the origin \(0\), and \([\Z^n:\Z A]\) is the index of \(\Z A\) in \(\Z^n\).
For the matrix \(A\) in \eqref{eq:matrix}, this convex hull is the simplex \(\operatorname{conv}(0, d\bm{e}_1, \ldots, d\bm{e}_n)\).
This simplex has volume \(d^n\) with respect to the measure \(d\mu\). Thus we obtain
\begin{equation}\label{eq:volume-index}
  \#(\Z A/d\Z^n) =[\Z A:d\Z^n] = \frac{d^n}{[\Z^n:\Z A]} = \operatorname{vol}(A).
\end{equation}

\begin{proposition}\label{prop:w-flat-consequences}
For all sufficiently large integers \(k\), the following statements hold.
\begin{enumerate}[(1)]
\item We have
\[
  \rank M_A(-\bm{\gamma}+kd\1) = \dim_\C \C[\theta_1, \ldots, \theta_m]/\fin_{\bm{w}}(H_A(-\bm{\gamma}+kd\1)) = \mathrm{vol}(A).
\]
\item The ideal \(\fin_{\bm{w}}(H_A(-\bm{\gamma}+kd\1))\) is radical, and the evaluation map
\begin{align*}
  \C[\theta_1, \ldots, \theta_m]/\fin_{\bm{w}}(H_A(-\bm{\gamma}+kd\1))
  &\longrightarrow \prod_{(\partial^{\bm{u}},[n])\in\mathcal{T}(\mathrm{in}_{\bm{w}}(I_A))}\C\\
  f &\longmapsto
  \bigl(f((-\bm{\gamma}+kd\1)^{(\partial^{\bm{u}},[n])})\bigr)_{(\partial^{\bm{u}},[n])\in\mathcal{T}(\mathrm{in}_{\bm{w}}(I_A))}
\end{align*}
is an isomorphism of \(\C\)-algebras.
\end{enumerate}
\end{proposition}

\begin{proof}
By Proposition~\ref{prop:w-flat} and \cite[Proposition~4.6.6]{SST}, the parameter \(-\bm{\gamma}+kd\1\) is \(\bm{w}\)-flat for all sufficiently large \(k\). (1) follows from \cite[Theorem~4.6.1, Remark~4.6.2]{SST}.

By Proposition~\ref{prop:w-flat}, the zero set of \(\fin_{\bm{w}}(H_A(-\bm{\gamma}+kd\1))\) consists precisely of the distinct points
\[
  (-\bm{\gamma}+kd\1)^{(\partial^{\bm{u}},[n])},\quad (\partial^{\bm{u}},[n])\in\mathcal{T}(\mathrm{in}_{\bm{w}}(I_A)).
\]
By Lemma~\ref{T=vol} and \eqref{eq:volume-index}, their number is \(\#(\Z A/d\Z^n)=\mathrm{vol}(A)\). Since the quotient algebra \(\C[\theta_1, \ldots, \theta_m]/\fin_{\bm{w}}(H_A(-\bm{\gamma}+kd\1))\) has dimension \(\mathrm{vol}(A)\) by (1), the ideal is radical and the evaluation map is an isomorphism.
\end{proof}

\subsection{Proof of the inequality (\ref{eq:main_inequality}) }
For a left ideal \(I\) of \(D_m\) such that \(D_m/I\) is holonomic, we denote the holonomic rank of \(D_m/I\) by \(\mathrm{rank}(I)\).

\begin{proposition}
Let \(\bm{w} \in \R^m\) be such that \(\operatorname{in}_{\bm{w}}(I_A)\) is a monomial ideal. Then, for every
\(\bm{\beta} \in \C^n\) and every \(\bm{u} \in \Z_{\geq 0}^m\),
\[
  \operatorname{rank} (H_A(\bm{\beta})+D_m\partial^{\bm{u}}) \leq \dim_{\C}\C[\theta_1, \ldots, \theta_m]/(\fin_{\bm{w}}(H_A(\bm{\beta})) + \langle[\theta]_{\bm{u}}\rangle).
\]
\end{proposition}

\begin{proof}
  Proposition \ref{prop:regularity} and \cite[Theorem~2.5.1]{SST} give
  \[
    \operatorname{rank}(H_A(\bm{\beta})+D_m\partial^{\bm{u}}) = \operatorname{rank}(\operatorname{in}_{(-\bm{w}, \bm{w})}(H_A(\bm{\beta})+D_m\partial^{\bm{u}})).
  \]

  For arbitrary \(\bm{\beta}\), \(D_m \operatorname{in}_{\bm{w}}(I_A)+D_m\langle A\bm{\theta} - \bm{\beta}\rangle \subset \operatorname{in}_{(-\bm{w}, \bm{w})}(H_A(\bm{\beta}))\). It follows that
  \[
    (D_m \operatorname{in}_{\bm{w}}(I_A) + D_m\langle A\bm{\theta}-\bm{\beta}\rangle + D_m\partial^{\bm{u}})\subseteq \operatorname{in}_{(-\bm{w},\bm{w})}(H_A(\bm{\beta})+D_m\partial^{\bm{u}}).
  \]

  Let
  \(R = \C(\lambda_1, \ldots, \lambda_m) \otimes_{\C[\lambda_1, \ldots, \lambda_m]} D_m\). Since \(\operatorname{in}_{\bm{w}}(I_A)\) is a monomial ideal, the identity \([\theta]_{\bm{v}}=\lambda^{\bm{v}}\partial^{\bm{v}}\) and the invertibility of \(\lambda^{\bm{v}}\) in \(R\) give \(R\operatorname{in}_{\bm{w}}(I_A)=R\widetilde{\operatorname{in}}_{\bm{w}}(I_A)\).
  From the preceding inclusion, using
  \(R\langle \partial^{\bm{u}} \rangle = R\langle \lambda^{\bm{u}}\partial^{\bm{u}} \rangle = R\langle [\theta]_{\bm{u}} \rangle\), we obtain
  \begin{align*}
    &\dim_{\C(\lambda_1, \ldots, \lambda_m)} R/R\left(\fin_{\bm{w}}(H_A(\bm{\beta})) + \langle[\theta]_{\bm{u}}\rangle \right)\\
    &= \dim_{\C(\lambda_1, \ldots, \lambda_m)}R/R(D_m \operatorname{in}_{\bm{w}}(I_A) + D_m\langle A\bm{\theta} - \bm{\beta}\rangle + D_m\partial^{\bm{u}})\\
    &\geq \dim_{\C(\lambda_1, \ldots, \lambda_m)}R/R\operatorname{in}_{(-\bm{w},\bm{w})}(H_A(\bm{\beta})+D_m\partial^{\bm{u}})\\
    &= \operatorname{rank}(H_A(\bm{\beta})+D_m\partial^{\bm{u}}).
  \end{align*}
  As a left \(\C(\lambda_1, \ldots, \lambda_m)\)-vector space , \(R \cong \C(\lambda_1, \ldots, \lambda_m)\otimes_{\C}\C[\theta_1, \ldots, \theta_m]\). Hence
  \begin{align*}
    \dim_{\C}\C[\theta_1, \ldots, \theta_m]/(\fin_{\bm{w}}(H_A(\bm{\beta})) + \langle[\theta]_{\bm{u}}\rangle) = \dim_{\C(\lambda_1, \ldots, \lambda_m)} R/R\left(\fin_{\bm{w}}(H_A(\bm{\beta})) + \langle[\theta]_{\bm{u}}\rangle \right).
  \end{align*}
  This proves the assertion.
\end{proof}

\begin{proposition}\label{prop:inequality}
  For a sufficiently large integer \(k\), set
  \(\bm{u}_k = (\underbrace{2k, \ldots, 2k}_{n}, 0, \ldots, 0)\in\Z_{\geq 0}^m\).
  Then
  \[
    \rank (H_A(-\bm{\gamma} + kd\1)+D_m\partial^{\bm{u}_k}) \leq \rank M_A(-\bm{\gamma} + kd\1) - \rank H^{n}_{\mathrm{Dw}}(Y/S)_{\overline{\bm{\gamma}}}.
  \]
\end{proposition}
\begin{proof}
  Let \(\bm{w}\) be a vector as in Lemma~\ref{generic_weight_vector} and take \(k\) sufficiently large so that Proposition~\ref{prop:w-flat-consequences} applies. Under the evaluation isomorphism in part~(2) of that proposition, \([\theta]_{\bm{u}_k}\) is sent to
  \begin{align*}
   \left(\prod_{i=1}^n\prod_{l=0}^{2k-1} \left(\frac{(-\bm{\gamma} - A\bm{u})_i + d(k-l)}{d}\right)\right)_{(\partial^{\bm{u}}, [n]) \in \mathcal{T}(\mathrm{in}_{\bm{w}}(I_A))}.
  \end{align*}
  where \((-\bm{\gamma} - A\bm{u})_i\) denotes the \(i\)-th component of \(-\bm{\gamma} - A\bm{u}\).
  Therefore, the dimension of \(\C[\theta_1, \ldots, \theta_m]/(\fin_{\bm{w}}(H_A(-\bm{\gamma} + kd\1))+\langle [\theta]_{\bm u_k} \rangle)\) as a \(\C\)-vector space is the cardinality of
  \[
    \{(\partial^{\bm{u}}, [n]) \in \mathcal{T}(\mathrm{in}_{\bm{w}}(I_A)) \mid  (\bm{\gamma} + A\bm{u})_i - d(k-l) = 0 \text{ for some } 1 \leq i \leq n, 0 \leq l \leq 2k-1\}.
  \]
  Its image under the bijection \(\mathcal{T}(\mathrm{in}_{\bm{w}}(I_A)) \rightarrow \Z A/d\Z^n\) of Lemma~\ref{T=vol} is contained in
  \begin{equation}
    \label{eq:the_set2}
    \{\overline{\bm{\beta}} \in \Z A/d\Z^n \mid \gamma_i + \beta_i \equiv 0 \bmod d \text{ for some } i = 1, \ldots, n\},
  \end{equation}
  and, by increasing \(k\) further if necessary, the image equals the set in \eqref{eq:the_set2}.
  This is the complement of the set (\ref{eq:totally_nonzero_set}) in \(\Z A/d\Z^n\). Hence
  \begin{align*}
    \rank \bigl(H_A(-\bm{\gamma} + kd\1)+D_m\partial^{\bm{u}_k}\bigr)
    &\leq \dim_\C \C[\theta_1, \ldots, \theta_m]/(\fin_{\bm{w}}(H_A(-\bm{\gamma} + kd\1))+\langle [\theta]_{\bm{u}_k} \rangle)\\
    &= \#\Z A/d\Z^n - \text{The number of elements of the set (\ref{eq:totally_nonzero_set})}\\
    &= \mathrm{vol}(A) - \rank H^{n}_{\mathrm{Dw}}(Y/S)_{\overline{\bm{\gamma}}}\\
    &= \rank M_A(-\bm{\gamma}+kd\1) - \rank H^{n}_{\mathrm{Dw}}(Y/S)_{\overline{\bm{\gamma}}},
  \end{align*}
  where the last equality follows from Proposition~\ref{prop:w-flat-consequences} (1).
\end{proof}

\section{Specialization to One-Parameter Families}
\subsection{The Generalized Dwork Family}\label{sec:generalized-dwork-family}
  Let \(d \geq n \geq 2\) and let \(\bm{a} =
  \begin{pmatrix}
    a_1 & \cdots & a_n
  \end{pmatrix}
  \) be a non-negative integer vector such that \(a_1 + \cdots + a_n = d\). Set
  \(I:=\{i\in\{1,\ldots,n\}\mid a_i\neq0\}\), and put
  \[
    g := \gcd(a_1,\ldots,a_n), \quad d' := d/g, \quad a_i' := a_i/g,\quad C_{\bm{a}}:=\frac{(-d)^{d'}}{\prod_{i \in I} a_i^{a_i'}}.
  \]

  We study the Gauss--Manin connection of the one-parameter family of projective hypersurfaces defined by the homogeneous equation
  \[X_1^d + \cdots + X_n^d + \lambda X_1^{a_1}\cdots X_n^{a_n}=0.\]
We refer to this family as the generalized Dwork family.
  \begin{proposition}
The projective hypersurface defined by the equation
\[
  X_1^d+\cdots+X_n^d+\lambda X_1^{a_1}\cdots X_n^{a_n} = 0 \quad (\lambda \in \C)
\]
is nonsingular if and only if
\[
\lambda^{d'}\neq C_{\bm{a}}.
\]
\end{proposition}

\begin{proof}
  Assume that the hypersurface defined by \(F = X_1^d+\cdots+X_n^d+\lambda X_1^{a_1}\cdots X_n^{a_n}\) has a singular point \(x = [x_1: \cdots : x_n] \in \PP^{n-1}_\C(\C)\). Then \(\lambda \neq 0\), because the Fermat hypersurface is nonsingular. For each \(i = 1, \ldots, n\), we have
  \[
x_i\frac{\partial F}{\partial X_i}(x) = dx_i^d + a_i \lambda x_1^{a_1}\cdots x_n^{a_n} = 0,
\]
and hence
  \begin{align*}
    (x_1^{a_1}\cdots x_n^{a_n})^{d'} &=\prod_{i\in I}\left(-\frac{a_i}{d}\lambda x_1^{a_1}\cdots x_n^{a_n} \right)^{a_i'} = C_{\bm{a}}^{-1}\lambda^{d'}(x_1^{a_1}\cdots x_n^{a_n})^{d'}.
  \end{align*}
    Since \(x_1^{a_1}\cdots x_n^{a_n} \neq 0\), otherwise \(dx_i^d = -a_i \lambda x_1^{a_1}\cdots x_n^{a_n} = 0\) for all \(i\), we obtain
  \begin{equation}
    \label{eq:condition}
    \lambda^{d'}=C_{\bm{a}}.
  \end{equation}

Conversely, assume that \(\lambda\) satisfies \eqref{eq:condition}.
Choose a \(d\)-th root \(\alpha_i\) of \(a_i\) for each \(i \in I\). By assumption,
\[
\left(\frac{-d}{\lambda\prod_{i \in I}\alpha_i^{a_i}}\right)^{d'}=1.
\]

Since \(\gcd(d,a_1,\ldots,a_n)=g\), the homomorphism
\[
  \mu_d^n\longrightarrow\mu_d;\quad (\zeta_i)\longmapsto\prod_{i=1}^n\zeta_i^{a_i}
\]
has image \(\mu_{d'}\). Therefore, there exist \(\zeta_i\in \mu_d\) such that \(\prod_{i = 1}^{n}\zeta_i^{a_i} = -d/(\lambda\prod_{i \in I}\alpha_i^{a_i})\). Now we define
\[
x_i=
\begin{cases}
\alpha_i\zeta_i & (a_i \neq 0),\\
0 & (a_i = 0).
\end{cases}
\]
Then for every \(i\),
\[
x_i\frac{\partial F}{\partial X_i}(x) =  d x_i^d+\lambda a_i \prod_{i \in I}x_i^{a_i} = da_i+\lambda a_i\left(-\frac{d}{\lambda}\right)=0.
\]
If \(a_i>0\), we have \(x_i\neq0\), so \(\partial F/\partial X_i(x)=0\). Euler's relation also gives \(F(x)=0\), so \(x = [x_1: \cdots : x_n]\) is a singular point.
\end{proof}

  Let \(Y/S\) be the family of projective hypersurfaces over \(S = \spec \C[\lambda_1,\ldots,\lambda_{n+1}]\) associated with the matrix
  \[
    A=
\begin{pmatrix}
d & 0 & \cdots & 0 & a_1\\
0 & d & \cdots & 0 & a_2\\
\vdots & \vdots & \ddots & \vdots & \vdots\\
0 & 0 & \cdots & d & a_n
\end{pmatrix}
  \]
  that is, the family defined by the homogeneous polynomial
  \[F_A = \lambda_1X_1^d + \cdots + \lambda_nX_n^d + \lambda_{n+1} X_1^{a_1}\cdots X_n^{a_n}.\]

\subsection{Restriction of GKZ Systems}\label{sec:restriction-one-parameter}
  Let \(\mathbb{A}^1_\lambda=\spec\C[\lambda]\), and let
  \[
    \iota:\mathbb{A}^1_\lambda\hookrightarrow S;\quad \lambda\longmapsto(1,\ldots,1,\lambda)
  \]
  be the closed immersion whose image is defined by \(\lambda_1 = \cdots = \lambda_n = 1\).
  \begin{proposition}\label{prop:noncharacteristic-restriction}
    For any \(\bm{\beta} \in \C^{n}\), \(M_A(\bm{\beta})\) is non-characteristic with respect to \(\iota\).
  \end{proposition}
  \begin{proof}
    The toric ideal of \(A\) is generated by \(\partial_1^{a_1/g} \cdots \partial_n^{a_n/g} - \partial_{n+1}^{d/g}\).

    For every \(\bm{\beta}\), the characteristic variety of \(M_A(\bm{\beta})\)
    is contained in the common zero locus of the principal symbols
    \[
      \begin{cases}
        \xi_1^{a_1/g}\cdots\xi_n^{a_n/g} - \xi_{n+1}^{d/g} = 0\\
        d\lambda_i\xi_i + a_i\lambda_{n+1}\xi_{n+1} = 0 \quad (i = 1,...,n)
      \end{cases}
    \]
    in the cotangent space \(T^*S = \spec \C[\lambda_1, \ldots,\lambda_{n+1}, \xi_1, \ldots, \xi_{n+1}]\).

    Since the conormal bundle of \(\iota(\mathbb{A}^1_\lambda)\) is defined by
        \[
      \begin{cases}
        \xi_{n+1} = 0\\
        \lambda_i = 1 \quad (i = 1,...,n),
      \end{cases}
    \]
    the Euler symbols force \(\xi_i=0\) for \(i=1,\ldots,n\) on this
    intersection. Thus the characteristic variety meets the conormal
    bundle only in the zero section, as required.
  \end{proof}

Write \(\partial_\lambda=d/d\lambda\) and \(\theta_\lambda=\lambda\partial_\lambda\), and let \(V=\spec\C[\lambda,1/(\lambda^{d'}-C_{\bm{a}})]\). Then \(Y\times_S\mathbb{A}^1_\lambda \rightarrow \mathbb{A}^1_\lambda\) is smooth over \(V\). We write \(j:V\setminus\{0\}\hookrightarrow V\).
The morphism
\begin{align*}
  h:V\setminus\{0\} &\longrightarrow \mathbb{G}_m\setminus\{1\} = \spec \C[t, (t(t-1))^{-1}]\\
   \lambda &\longmapsto t=C_{\bm a}^{-1}\lambda^{d'}
\end{align*}
is finite \'{e}tale. With \(\theta_t=t\partial_t\), we have
\(\theta_\lambda=d'\theta_t\) after pullback.

\begin{definition}\label{def:hypergeometric-module}
  For \(N\in\Z_{\geq0}\) and parameter tuples \(\bm{\mu}=(\mu_1,\ldots,\mu_N)\) and
  \(\bm{\nu}=(\nu_1,\ldots,\nu_N)\) in \(\C^N\), define
  \[
    \hyp(\bm{\mu},\bm{\nu}) :=\prod_{j=1}^N(\theta_t+\nu_j-1) - t\prod_{i=1}^N(\theta_t+\mu_i).
  \]
  The associated hypergeometric \(\D\)-module on
  \(\mathbb{G}_m\setminus\{1\}\) is
  \[
    \mathcal{H}(\bm{\mu}; \bm{\nu}) := \D_{\mathbb{G}_m\setminus\{1\}}/\D_{\mathbb{G}_m\setminus\{1\}}\hyp(\bm{\mu}; \bm{\nu}).
  \]
\end{definition}

For \(\bm{\beta}\in\C^n\), we use the notation
\begin{align*}
\mathcal{H}_{\bm{\beta}} &:=\mathcal{H}\left(\left(\frac{-\beta_i+d\ell}{d a_i'}\;\middle|\;i\in I,\ 0\leq\ell<a_i'\right); \left(\frac1{d'},\frac2{d'},\ldots,1\right)\right)\\
&= \D_{\mathbb{G}_m\setminus\{1\}}/\D_{\mathbb{G}_m\setminus\{1\}}\left(\prod_{j=1}^{d'}\left(\theta_t + \frac{j}{d'} - 1\right) - t\prod_{i\in I}\prod_{\ell=0}^{a_i'-1} \left(\theta_t + \frac{-\beta_i+d\ell}{d a_i'}\right)\right).
\end{align*}

\begin{proposition}\label{prop:restricted-gkz-operator}
  For every \(\bm{\beta} = (\beta_1, \ldots, \beta_n) \in \C^n\), there is an isomorphism
\begin{align*}
  (\iota^+M_A(\bm{\beta}))|_{V\setminus\{0\}} &\cong h^*\mathcal{H}_{\bm{\beta}}.
\end{align*}
\end{proposition}

\begin{proof}
  By Proposition~\ref{prop:noncharacteristic-restriction}, the inverse
  image is concentrated in degree zero. As a left \(\D_{\mathbb{A}^1_\lambda}\)-module,
  \[
    \D_{\mathbb{A}^1_\lambda\rightarrow S}\cong\D_{\mathbb{A}^1_\lambda}\otimes_\C\C[\partial_1,\ldots,\partial_n].
  \]
  Quotienting by the Euler relations gives
  \[
    \D_{\mathbb{A}^1_\lambda\to S}/\sum_{i=1}^n\D_{\mathbb A^1_\lambda\to S}(E_i-\beta_i) \cong \D_{\mathbb{A}^1_\lambda}; \quad 1\otimes\partial_i\longmapsto\frac{\beta_i-a_i\theta_\lambda}{d}.
  \]
  Since \(I_A\) is generated by \(\partial_{n+1}^{d'} - \prod_i\partial_i^{a_i'}\), we obtain
  \[
    \iota^+M_A(\bm{\beta}) \cong \D_{\mathbb{A}^1_\lambda}/\D_{\mathbb{A}^1_\lambda}\left(\partial_\lambda^{d'} - C_{\bm{a}}^{-1}\prod_{i\in I}\prod_{\ell=0}^{a_i'-1}\left(\theta_\lambda+\frac{-\beta_i+d\ell}{a_i}\right)\right).
  \]
  Since \(\lambda\) is invertible on \(V\setminus\{0\}\), multiplying by \(\lambda^{d'}\) gives
  \begin{align*}
    &\D_{V\setminus\{0\}}\left(\partial_\lambda^{d'} - C_{\bm{a}}^{-1}\prod_{i\in I}\prod_{\ell=0}^{a_i'-1} \left(\theta_\lambda+\frac{-\beta_i+d\ell}{a_i}\right)\right)\\
    &= \D_{V\setminus\{0\}}\left(\prod_{j = 0}^{d'-1}(\theta_\lambda - j)-C_{\bm{a}}^{-1}\lambda^{d'}\prod_{i\in I}\prod_{\ell=0}^{a_i'-1}\left(\theta_\lambda+\frac{-\beta_i+d\ell}{a_i}\right)\right).
  \end{align*}
  On the other hand,
  \[
    h^*\mathcal{H}_{\bm{\beta}}\cong \D_{V\setminus\{0\}}/\D_{V\setminus\{0\}}\left(\prod_{j=1}^{d'}\left(\frac{\theta_\lambda}{d'}+\frac{j}{d'}-1\right) - C_{\bm{a}}^{-1}\lambda^{d'}\prod_{i\in I}\prod_{\ell=0}^{a_i'-1}\left(\frac{\theta_\lambda}{d'}+\frac{-\beta_i+d\ell}{d a_i'}\right)
    \right).
  \]
  Hence \(\iota^+M_A(\bm{\beta})|_{V\setminus\{0\}}\) is isomorphic to \(h^*\mathcal{H}_{\bm{\beta}}\).
\end{proof}

\subsection{The Gauss--Manin Connection of the Generalized Dwork Family}
\begin{definition}
  For two parameter vectors \(\bm{\mu} = (\mu_1, \ldots, \mu_N), \bm{\nu} = (\nu_1, \ldots, \nu_N) \in \C^N\), we say a pair of \(\mu_i\) and \(\nu_j\) is a resonant pair if \(\mu_i \equiv \nu_j \bmod \Z\). Let \(\bm{\mu}', \bm{\nu}'\) be the parameter vectors obtained by removing all resonant pairs. Then we define
  \[
    \Cancel\mathcal{H}(\bm{\mu}; \bm{\nu}) := \mathcal{H}(\bm{\mu}';\bm{\nu}').
  \]
\end{definition}

For any parameters \(\bm{\mu}, \bm{\nu} \in \C^N\), since
\[
  \hyp(\bm{\mu},\bm{\nu})(\theta_t + \mu_i - 1)=(\theta_t + \mu_i - 1)\hyp(\bm{\mu}-\bm{e}_i,\bm{\nu}),
\]
right multiplication by \(\theta_t+\mu_i-1\) defines
\begin{equation}\label{eq:hypergeometric-contiguity}
  \cdot(\theta_t + \mu_i -1):\mathcal{H}(\bm{\mu}; \bm{\nu})
  \longrightarrow\mathcal{H}(\bm{\mu}-\bm{e}_i;\bm{\nu}).
\end{equation}
There are two cases:
\begin{enumerate}
  \item If \(\mu_i\neq\nu_j\) for every \(j\), then (\ref{eq:hypergeometric-contiguity}) is an
  isomorphism by \cite[Proposition~2.5]{BH}.
  \item If \(\mu_i=\nu_j\) for some \(j\), the image of (\ref{eq:hypergeometric-contiguity}) is isomorphic to
  \[
    \mathcal{H}((\mu_1, \ldots, \widehat{\mu_i}, \ldots, \mu_N);(\nu_1, \ldots, \widehat{\nu_j}, \ldots, \nu_N)).
  \]
\end{enumerate}

Assume that \((\mu_i, \nu_j)\) is a resonant pair such that \(|\mu_i - \nu_j|\) is minimal among all resonant pairs. If \(\delta = \mu_i - \nu_j \geq 0\), then the image of a chain of contiguity operators
\[
  \mathcal{H}(\bm{\mu}; \bm{\nu}) \xrightarrow[\cdot(\theta_t + \mu_i -1)]{\sim} \cdots \xrightarrow[\cdot(\theta_t + \mu_i - \delta)]{\sim}\mathcal{H}(\bm{\mu} - \delta\bm{e}_i; \bm{\nu}) \xrightarrow[\cdot(\theta_t + \mu_i-\delta -1)]{}\mathcal{H}(\bm{\mu} - (\delta + 1)\bm{e}_i; \bm{\nu})
\]
is isomorphic to \( \mathcal{H}((\mu_1, \ldots, \widehat{\mu_i}, \ldots, \mu_N);(\nu_1, \ldots, \widehat{\nu_j}, \ldots, \nu_N))\). When \(\delta = \mu_i - \nu_j < 0\), the image of
\[
  \mathcal{H}(\bm{\mu} - \delta\bm{e}_i; \bm{\nu}) \xrightarrow[\cdot(\theta_t + \mu_i - \delta -1)]{} \mathcal{H}(\bm{\mu} - (\delta + 1)\bm{e}_i; \bm{\nu}) \xrightarrow[\cdot(\theta_t + \mu_i -\delta - 2)]{\sim} \cdots \xrightarrow[\cdot(\theta_t + \mu_i)]{\sim}\mathcal{H}(\bm{\mu}; \bm{\nu})
\]
is isomorphic to \( \mathcal{H}((\mu_1, \ldots, \widehat{\mu_i}, \ldots, \mu_N);(\nu_1, \ldots, \widehat{\nu_j}, \ldots, \nu_N))\). Repeating this procedure removes all remaining resonant pairs. Therefore, \(\Cancel \mathcal{H}(\bm{\mu}; \bm{\nu})\) is obtained as the image of a chain of the contiguity operators.\medskip

The following result generalizes \cite[Corollary~8.5]{K09} and \cite{N}. The same result is also proved in \cite[\S 5]{AM} in the case \(g=1\).

\begin{theorem}\label{thm:dwork-hypergeometric}
Let \(j:V\setminus\{0\} \rightarrow V\) be the inclusion. When \(H^{n-2}_{\dR}(Y_V/V)^{\mathrm{prim}}_{\overline{\bm{\gamma}}} \neq 0\),
\begin{align*}
    H^{n-2}_{\dR}(Y_V/V)^{\mathrm{prim}}_{\overline{\bm{\gamma}}}|_{V\setminus\{0\}}\cong h^*(\Cancel\mathcal H_{-\bm{\gamma}}),\\
    H^{n-2}_{\dR}(Y_V/V)^{\mathrm{prim}}_{\overline{\bm{\gamma}}} \cong j_{!+} h^*(\Cancel\mathcal H_{-\bm{\gamma}}).
\end{align*}
\end{theorem}

\begin{proof}
By Theorem~\ref{thm:gauss-manin-gkz} and
Proposition~\ref{prop:contiguity-image},
\[
  H^{n-2}_{\dR}(Y_V/V)^{\mathrm{prim}}_{\overline{\bm{\gamma}}} \cong \iota^+\operatorname{Im}\!\left(M_A(-\bm{\gamma}-kd \1) \xrightarrow{\cdot(\partial_1\cdots\partial_n)^{2k}} M_A(-\bm{\gamma}+kd \1)\right)|_V
\]
for sufficiently large \(k\).
Under the isomorphism of Proposition~\ref{prop:restricted-gkz-operator}, the restriction of the contiguity morphism \(\cdot\partial_i:M_A(\bm{\beta})\to M_A(\bm{\beta}+d\bm{e}_i)\) is identified with the pullback by \(h\) of
\[
  \cdot\frac{\beta_i+d-d a_i'\theta_t}{d}:\mathcal{H}_{\bm{\beta}}\longrightarrow \mathcal{H}_{\bm{\beta}+d\bm{e}_i}.
\]

Non-characteristic restriction
(Proposition~\ref{prop:noncharacteristic-restriction}) and
\'{e}tale pullback preserve images. Hence the restriction of \(H^{n-2}_{\dR}(Y_V/V)^{\mathrm{prim}}_{\overline{\bm{\gamma}}}\) to \(V\setminus\{0\}\) is the pullback by \(h\) of
\begin{equation}\label{eq:contiguity_image}
  \operatorname{Im}\!\left(\cdot\prod_{i\in I}\prod_{m=-k+1}^{k}\left(\theta_t+\frac{\gamma_i-md}{d a_i'}\right) : \mathcal{H}_{-\bm{\gamma}-kd\1} \longrightarrow \mathcal{H}_{-\bm{\gamma}+kd\1}\right).
\end{equation}

All resonant pairs of \(\mathcal{H}_{-\bm{\gamma}}\) are removed in the chain of contiguity operators  for sufficiently large \(k\), hence the image in \eqref{eq:contiguity_image} is isomorphic to \(\Cancel\mathcal H_{-\bm{\gamma}}\). This proves that
\[
  H^{n-2}_{\dR}(Y_V/V)^{\mathrm{prim}}_{\overline{\bm{\gamma}}}|_{V\setminus\{0\}} \cong h^*(\Cancel\mathcal H_{-\bm{\gamma}}).
\]

Since \(H^{n-2}_{\dR}(Y_V/V)^{\mathrm{prim}}_{\overline{\bm{\gamma}}}\) is an \(\mathcal O_V\)-coherent \(\D_V\)-module,
\[
    H^{n-2}_{\dR}(Y_V/V)^{\mathrm{prim}}_{\overline{\bm{\gamma}}} \cong j_{!+}\left(H^{n-2}_{\dR}(Y_V/V)^{\mathrm{prim}}_{\overline{\bm{\gamma}}}|_{V\setminus\{0\}}\right) \cong j_{!+}h^*(\Cancel\mathcal H_{-\bm{\gamma}}).
\]
\end{proof}

\begin{example}[{\cite[Example~5.5]{Klo}}]\label{ex:quartic-monomial-deformation}
  Consider the family of genus \(3\) curves defined by
  \[
    X^4+Y^4+Z^4+\lambda X^2Y^2=0
  \]
  over \(V=\spec\C[\lambda,1/(\lambda^2-4)]\). Here \(n=3\), \(d=4\), \(\bm{a}=(2,2,0)\), and
  \[
    g=2,\quad d'=2,\quad (a_1',a_2',a_3')=(1,1,0), \quad C_{\bm{a}}=4.
  \]
  Thus \(h:V\setminus\{0\}\to\mathbb G_m\setminus\{1\}\) is given by
  \(t=\lambda^2/4\). The nonzero components of \(H^{1}_{\dR}(Y_V/V)^{\mathrm{prim}}\) correspond to the classes in \(H_{3,4}/\Z A\cong\widehat{\Gamma_A/\Delta}\) represented by
\[
  (1,1,2),\quad (1,2,1),\quad (2,1,1), \quad (2,3,3),\quad (3,2,3).
\]
For \(\bm{\gamma} = (1,2,1)\),
\[
  \mathcal{H}_{-\bm{\gamma}} = \mathcal{H}\left(\left(\frac{1}{4}, \frac{1}{2}\right);\left(\frac{1}{2}, 1\right)\right).
\]
By Theorem~\ref{thm:dwork-hypergeometric}, \(H^{1}_{\dR}(Y_V/V)^{\mathrm{prim}}_{\overline{\bm{\gamma}}}|_{V\setminus\{0\}}\) is isomorphic to the pullback by \(h\) of
\[
  \Cancel\mathcal{H}_{-\bm{\gamma}} = \mathcal{H}\left(\left(\frac{1}{4}\right);\left(1\right)\right) = \D/\D(\theta_t-t(\theta_t+\tfrac{1}{4})).
\]
\(\Cancel\mathcal{H}_{-\bm{\gamma}}\) for the other \(\bm{\gamma}\) are listed in Table~\ref{tab:quartic-hypergeometric-modules}.
\begin{table}[htbp]
  \centering
  \(
    \begin{array}{c|c|c}
      \bm{\gamma}
      & \Cancel\mathcal H_{-\bm{\gamma}} & \text{rank} \\\hline
      (2,1,1), (1,2,1)
      & \D/\D(\theta_t - t(\theta_t+\tfrac{1}{4})) & 1 \\ (2,3,3), (3,2,3)
      & \D/\D(\theta_t-t(\theta_t+\tfrac{3}{4})) & 1 \\ (1,1,2)
      & \D/\D(\theta_t(\theta_t-\tfrac{1}{2})-t(\theta_t+\tfrac{1}{4})^2) & 2
    \end{array}
  \)
  \caption{Hypergeometric \(\D\)-modules associated with the quartic family.}
  \label{tab:quartic-hypergeometric-modules}
\end{table}
\end{example}

\begin{remark}
  Although we have worked over \(\C\) throughout this paper, the same results in fact hold over any field of characteristic zero containing the \(d\)-th roots of unity. Therefore, by base change to a suitable \(p\)-adic field, one can recover, as in \cite[\S 6]{N}, some of the results of \cite{Klo} concerning the Frobenius action on \(p\)-adic cohomology.
\end{remark}

\subsection*{Acknowledgments}
The author would like to thank Saiei-Jaeyeong Matsubara-Heo for suggesting an idea for the proof of Proposition \ref{prop:inequality}. This work was supported by JST K Program Grant Number JPMJKP24U2.
\bibliography{ref}
\bibliographystyle{plain}

\bigskip
\noindent
\begin{minipage}{\textwidth}
\raggedright
\textbf{Ryo Negishi}\\
{Department of Mathematics}\\
{Rikkyo University}\\
{Tokyo 171-8501, Japan}\\
\texttt{5072908@rikkyo.ac.jp}
\end{minipage}

\end{document}